\documentclass{amsart}
\usepackage{amsthm, amssymb, latexsym, amsmath, color, mathrsfs, caption, mathtools}
\usepackage[utf8]{inputenc}

\usepackage{enumitem}

\usepackage{multirow}
\usepackage{float} 
\usepackage{tikz-cd} 
\usepackage[all]{xy}

\usepackage{hyperref}
\hypersetup{
colorlinks=true, 
linkcolor=blue, 
filecolor=magenta, 
urlcolor=cyan, 
}
\usepackage{verbatim}
\usepackage{fancyvrb}
\usepackage{xcolor}
\usepackage{listings}

\definecolor{codegreen}{rgb}{0, 0.6, 0}
\definecolor{codegray}{rgb}{0.5, 0.5, 0.5}
\definecolor{codepurple}{rgb}{0.58, 0, 0.82}
\definecolor{backcolour}{rgb}{0.95, 0.95, 0.92}

\lstdefinelanguage{Macaulay2}
{
xleftmargin=.2in, 
xrightmargin=.2in, 
basicstyle={\ttfamily}, 
keywordstyle={\color{blue}}, 
commentstyle={\color{codegreen}}, 
stringstyle={\color{red!40!black}}, 
rulecolor=\color{yellow}, 
basewidth={1.2ex}, 
sensitive=false, 
morecomment=[l]{--}, 
morecomment=[s]{-*}{*-}, 
morestring=[b]", 
escapechar={`}, 
escapebegin={\rmfamily}, 
morekeywords={load, random, degree, genus, topComponents, ideal, Ext, minors, quotient, intersect, map, kernel, preimage, codim, sheaf, matrix, hilbertPolynomial, Projective, false, sheafExt, ann, cooker, flatten, gens, entries, basis, apply}
}
\lstalias{Macaulay2output}{Macaulay2}

\newcommand{\Z}{{\mathbb Z}}

\newcommand{\p}[1]{{\mathbb{P}^{#1}}}

\newcommand{\op}[1]{{\mathcal O}_{\mathbb{P}^{#1}}}
\newcommand{\opn}{{\mathcal O}_{\mathbb{P}^n}}
\newcommand{\tp}[1]{{\rm T}{\mathbb{P}^{#1}}}

\newcommand{\inhom}{{\mathcal H}{\it om}}
\newcommand{\inext}{{\mathcal E}{\it xt}}
\newcommand{\intor}{{\mathcal T}{\it or}}

\newcommand{\OO }{\mathcal{O} }
\newcommand{\calo}{\mathcal{O}}
\newcommand{\sI}{\mathscr{I}}

\DeclareMathOperator{\Hom}{Hom}
\DeclareMathOperator{\Ext}{Ext}
\DeclareMathOperator{\homd}{hom} 
\DeclareMathOperator{\extd}{ext}
\DeclareMathOperator{\Aut}{Aut}
\DeclareMathOperator{\coker}{coker}

\DeclareMathOperator{\rk}{{rk}}

\DeclareMathOperator{\Pic}{{Pic}}
\DeclareMathOperator{\gl}{{GL}}

\newcommand{\hilb}{\rm Hilb}

\newcommand{\calf}{{\mathcal F}}
\newcommand{\calr}{{\mathcal R}}
\newcommand{\calh}{{\mathcal H}}
\newcommand{\calc}{{\mathcal C}}

\newcommand{\PE}{ \mathbb{P}(\mathcal{E})}

\newtheorem{theorem}{Theorem}

\newtheorem{proposition}[theorem]{Proposition}

\newtheorem{lemma}[theorem]{Lemma}

\theoremstyle{definition}
\newtheorem{remark}[theorem]{Remark}
\newtheorem{example}[theorem]{Example}

\makeatletter
\def\subsubsection{\@startsection{subsubsection}{3}%
\z@{.5\linespacing\@plus.7\linespacing}{-.5em}%
{\normalfont\bfseries}}
\def\paragraph{\@startsection{paragraph}{4}%
 \z@{.5\linespacing\@plus.5\linespacing}{-\fontdimen2\font}%
 {\normalfont\bfseries}}
\def\subparagraph{\@startsection{subparagraph}{5}%
 \z@{.3\linespacing\@plus.3\linespacing}{-\fontdimen2\font}%
 {\normalfont\bfseries}}
\makeatother

\makeatletter
\@namedef{subjclassname@2020}{%
 \textup{2020} Mathematics Subject Classification}
\makeatother

\title[rank-two Reflexive Sheaves on $\p3$ with $c_2=4$]{Rank-two reflexive sheaves on the projective space with second Chern class equal to four}

\usepackage[foot]{amsaddr}

\author[M. Jardim]{Marcos Jardim}
\email{jardim@unicamp.br}

\author[A. Muniz]{Alan Muniz}
\email[A. Muniz]{alan.nmuniz@ufpe.br, alannmuniz@gmail.com}

\address[MJ, AM]{Universidade Estadual de Campinas (UNICAMP) \\ Instituto de Matemática, Estatística e Computação Científica (IMECC) \\ Departamento de Matem\'atica \\
Rua S\'ergio Buarque de Holanda, 651\\ 13083-859 Campinas-SP, Brazil}

\address[AM]{Departamento de Matem\'atica \\ Centro de Ci\^encias Exatas e da Natureza \\ Universidade Federal de Pernambuco \\ Recife - PE, CEP 50740-560, Brazil }

\date{September 2024}



\begin{document}

\begin{abstract}
We study rank-two reflexive sheaves on $\p3$ with $c_2 =4$, expanding on previous results for $c_2\le3$. We show that every spectrum not previously ruled out is realized. Moreover, moduli spaces are studied and described in detail for $c_1=-1$ or $0$ and $c_3\ge8$.
\end{abstract}

\maketitle

\tableofcontents
\section{Introduction}

Stable rank-$2$ reflexive sheaves have been the subject of an influential article by Hartshorne \cite{H2} more than 40 years ago, and numerous authors have studied them since then. Most relevant for the present paper, Chang provided in \cite{Ch} a complete description of the moduli spaces when the second Chern class $c_2\leq 3$. While Mir\'o-Roig \cite{MR} and Chang \cite{Ch3} established existence results for sheaves with $c_2\geq 4$ and a given third Chern class. This work aims to study moduli spaces of stable rank-$2$ reflexive sheaves with $c_2=4$.

Beyond Chern classes, rank-$2$ reflexive sheaves admit finer numerical invariants, also introduced by Hartshorne in \cite{H2}, called the \emph{spectrum}. This is a multi-set of integers satisfying some strict conditions, cf. Theorem \ref{thm: boundspec} below. 
Our first step is to list all possible spectra for $c_2 = 4$, cf. Tables \ref{tab: spec c1=0} and \ref{tab: spec c1=-1} below. Two of these spectra (marked with asterisks) were known to be unrealizable. In particular, there are no rank-$2$ reflexive sheaves $F$ with $(c_1,c_2,c_3)=(-1,4,14)$, cf. \cite[Theorem A]{MR} or \cite[Proposition 2]{Ch3}. As part of our work, we exhibit sheaves evincing each other spectra.

Let $\calr(c_1, c_2, c_3)$ denote the moduli space of stable rank-$2$ reflexive sheaves $F$ with Chern classes $(c_1, c_2, c_3)$. For $c_2=4$, essentially two situations have been studied in the literature: the case of locally free sheaves (i.e., $c_3 = 0$) and the case of maximal $c_3$ (that is, $c_3=14$ when $c_1=0$ and $c_3=16$ when $c_1=-1$). Let us revise the available results. The case of vector bundles was described by Barth \cite{Barth1981} and Chang \cite{Ch2} (for $c_1=0$) and by B\u{a}nic\u{a} and Manolache \cite{BM} (for $c_1=-1$). These authors have shown that each of the moduli spaces $\calr(c_1,4,0)$ consists of exactly two irreducible components; both are of expected dimension 29 in the case $c_1=0$, while one of the components is oversized in the case $c_1=-1$. For maximal $c_3$, the moduli spaces were described by Hartshorne in \cite{H2} and Chang \cite{Ch3}. More precisely, both $\calr(0, 4, 14)$ and $\calr(-1, 4, 16)$ are irreducible, nonsingular, and rational of dimension 29, cf. \cite[Theorem 5]{Ch3} and \cite[Theorem 9.2]{H2}, respectively. Note that the former has the expected dimension, while the latter is oversized. This paper examines the remaining cases for $\calr(c_1,4, c_3)$. To be precise, we prove the following results.

\begin{theorem} \label{thm: 0}
Set $c_1=0$.
\begin{enumerate}
\item For $1\le c_3\le6$, the moduli scheme $\calr(0,4,c_3)$ has a generically reduced and irreducible component of expected dimension $29$.
\item The moduli scheme $\calr(0,4,8)$ is generically smooth, unirational of dimension $29$. Moreover, the reduction $\calr(0,4,8)_{\rm red}$ is irreducible. 
\item The moduli scheme $\calr(0,4,10)$ is smooth, irreducible, unirational of dimension $29$.
\item The moduli scheme $\calr(0,4,12)$ possesses two irreducible components
\begin{itemize}
 \item $\calr(0,4,12)_0$ is smooth, unirational of dimension $29$;
 \item $\calr(0,4,12)_1$ is non-reduced of dimension 29, and $\big(\calr(0,4,12)_1\big)_{\rm red}$ is smooth and rational.
\end{itemize}
\end{enumerate}
\end{theorem}

Further details on the structure of $\calr(0,4,8)$, $\calr(0,4,10)$, and $\calr(0,4,12)$ are provided in Theorem \ref{thm: 048}, Theorem \ref{thm: 0410}, and Theorem \ref{thm: 0412}, respectively. Remarkably, $\calr(0,4,12)_1$ is non-reduced at a general point; the reduction $(\calr(0,4,12)_1)_{\rm red}$ has already been described by Chang in \cite[Theorem 10]{Ch3}. To the extent of our knowledge, the moduli spaces $\calr(c_1,c_2,c_3)$ previously known to have generically non-reduced irreducible components are: $\calr(0,13,66)$, $\calr(0, 13, 74)$, and $\calr(-1, 14, 88)$, by Kleppe in \cite[p.1138]{Kleppe}; and $\calr(0,14,0)$ by Lavrov in \cite{Lavrov}. Both Kleppe's and Lavrov's examples are based on the notorious non-reduced component of the Hilbert Scheme $\hilb^{14,24}(\p3)$, of curves of degree $14$ and genus $24$, described by Mumford in \cite{MR0148670}. Note that the value of $c_2$ in our example is much lower; furthermore, our example is independent of Mumford's result. Finally, it is worth pointing out that we do not know whether $\calr(0,4,8)$ is reduced.

\begin{theorem} \label{thm: -1}
Set $c_1=-1$.
\begin{enumerate}
\item For $1\le c_3\le6$, the moduli scheme $\calr(-1,4,c_3)$ has a generically reduced and irreducible component of the expected dimension $27$.
\item The moduli scheme $\calr(-1,4,8)$ is generically smooth of dimension 27, singular along a subvariety of codimension 1. Furthermore, the reduction $(\calr(-1,4,8))_{\rm red}$ is irreducible.
\item The moduli scheme $\calr(-1,4,10)$ is integral, unirational of dimension 27, smooth away from codimension 2.
\item The moduli scheme $\calr(-1,4,12)$ is integral, smooth, unirational of dimension 27. 
\end{enumerate}
\end{theorem}

Further details on the structure of $\calr(-1,4,8)$, $\calr(-1,4,10)$ and $\calr(-1,4,12)$ are provided in Theorem \ref{thm: -148}, Theorem \ref{thm: -1,10} and Theorem \ref{thm: -1412}, respectively. The remarkable feature of Theorem \ref{thm: -1} is that the moduli schemes $\calr(-1,4,c_3)$ for $c_3=8,10,12$ are much simpler than their counterparts with $c_1=0$.

This paper is organized as follows. The initial sections \ref{sec: spectrum} and \ref{sec: serre} contain a revision of mostly known facts about the spectra for stable rank-$2$ reflexive sheaves and the Serre correspondence between such sheaves and space curves, with a focus on sheaves on $c_2=4$. The main new contribution here is classifying all spectra for stable rank-$2$ reflexive sheaves with $c_2=4$, cf. Tables \ref{tab: spec c1=0} and \ref{tab: spec c1=-1}. Sections \ref{sec: extremal} and \ref{sec: unobsheaves} establish some essential technical results. Sections \ref{sec: even-low} and \ref{sec: odd-low} present examples of sheaves with $c_3 \leq 6$, while the remaining sections provide a complete description of the moduli spaces $\calr(c_1,4,c_3)$ for $c_3\geq 8$. 

Some of our results were driven by empiric evidence gained through the Macaulay2 computer algebra system \cite{M2}. 
{These computations are included in ancillary files} \verb|refshsetup.m2| and \verb|examples.m2|, {available at \url{https://github.com/alannmuniz/refsheaves.git}.}


\subsection*{Acknowledgments}
MJ is supported by the CNPQ grant number 305601/2022-9, the FAPESP Thematic Project number 2018/21391-1, and the FAPESP-ANR project number 2021/04065-6. AM was supported by INCTmat/MCT/Brazil, CNPq grant number 160934/2022-2. {We warmly thank the anonymous referee for many important comments and suggestions.}


\section{Spectrum and cohomology of rank-\texorpdfstring{$2$}{2} reflexive sheaves} \label{sec: spectrum}

We start by collecting basic facts about the rank-$2$ reflexive sheaves on $\p3$, with Chern classes denoted by $c_1,c_2$ and $c_3$. We will assume that the sheaves are normalized, i.e., $c_1 \in \{-1, 0\}$, and that $c_2 = 4$. In this case, the Hirzebruch--Riemann--Roch Theorem reads as follows.

\begin{lemma}[HRR]\label{lem: HRR}
Let $F$ be a normalized rank-$2$ reflexive sheaf with second Chern class $c_2 = 4$. Then,
\[
\chi(F(l)) = \begin{cases} \displaystyle
\frac{1}{3}l^3 + 2l^2 - \frac{1}{3}l + \frac{c_3 }{2} -6, & \text{if } c_1 = 0; \\[5pt] \displaystyle
\frac{1}{3}l^3 + \frac{3}{2}l^2 - \frac{11}{6}l + \frac{c_3 }{2} -5, & \text{if } c_1 = -1.
\end{cases}
\]
In particular, $c_3$ is even.
\end{lemma}

Given a normalized rank-$2$ reflexive sheaf $F$ on $\p3$, such that $h^0(F(-1)) = 0$, there exists a unique list of integers $\{k_1, \dots, k_{c_2(F)}\}$, encoding partial information on its cohomology, called the \emph{spectrum} of $F$. It was first defined by Barth and Elencwajg \cite{MR0517517} for locally free sheaves and later extended to reflexive sheaves by Hartshorne \cite[Section 7]{H2}. The following properties characterize the spectrum of $F$: 
\begin{enumerate}
 \item[S1)] \label{h1spec} 
 $h^1(F(p)) = \sum_{i} h^0(\p1, \op1(k_i+p+1))$, for each $p\leq -1$;
 \item[S2)] \label{h2spec} 
 $h^2(F(p)) = \sum_{i} h^1(\p1, \op1(k_i+p+1))$, for each $p\geq -3$ if $c_1 = 0$ and $p\geq -2$ if $c_1 = -1$.
\end{enumerate}
Moreover, $c_3 =-2\sum k_i + c_1c_2$.

In the other direction, for given $c_1$, $c_2$, and $c_3$, one may determine all the possible spectra using the following criteria, cf. \cite[Theorem 7.5]{H2}.

\begin{theorem}[Hartshorne] \label{thm: boundspec}
Let $F$ be a normalized rank-$2$ reflexive sheaf such that $h^0(F(-1))= 0$ and let $\{k_i\}$ be its spectrum, then:
\begin{enumerate}
 \item If there exists $k>0$ in the spectrum, then $1, 2, \dots, k$ also occur in the spectrum;
 \item If there exists $k<-1$ in the spectrum, then $-1, -2, \dots, k$ also occur if $c_1 = 0$, and $-2, \dots, k$ also occur if $c_1 = -1$.
\end{enumerate}
In addition, if $F$ is stable, then:
\begin{enumerate}
 \item If there exists $k>0$ in the spectrum, then $0$ also occurs;
 \item If there exists $k<-1$ in the spectrum, then $-1$ also occurs and, if $c_1 = 0$ then, either $0$ also occurs or $-1$ occurs at least twice.
\end{enumerate}
\end{theorem}

Fixing $c_2 = 4$ and $c_1\in \{ -1, 0\}$, we get the following lists of spectra. Recall that $c_3 \leq c_2^2 + (1+c_1)(2-c_2)$ for a stable $F$, cf. \cite[Theorem 8.2]{H2}.

\begin{table}[H] \def\arraystretch{1.5}
\begin{tabular}{|c|c|}
\hline
$c_3$ & spectra \\ \hline\hline
\multirow{2}{*} 0 & $\{0, 0, 0, 0\}$ \\ \cline{2-2} 
 & $\{-1, 0, 0, 1\}$ \\ \hline
\multirow{2}{*}{2} & $\{-1, 0, 0, 0\}$ \\ \cline{2-2} 
 & $\{-1, -1, 0, 1\}$ \\ \hline
\multirow{2}{*}{4} & $\{-1, -1, 0, 0\}$ \\ \cline{2-2} 
 & $\{-2, -1, 0, 1\}$ \\ \hline
\multirow{2}{*}{6} & $\{-1, -1, -1, 0\}$ \\ \cline{2-2} 
 & $\{-2, -1, 0, 0\}$ \\ \hline
\end{tabular}\quad \quad
\begin{tabular}{|c|c|}\hline
$c_3$ & spectra \\ \hline\hline
\multirow{2}{*} 8 & $\{-1, -1, -1, -1\}$ \\ \cline{2-2} 
 & $\{-2, -1, -1, 0\}$ \\ \hline
\multirow{2}{*} {10} & $\{-2, -1, -1, -1\}$ \\ \cline{2-2}
 & $*\{-2, -2, -1, 0\}$ \\ \hline
\multirow{2}{*}{12} & $\{-3, -2, -1, 0\}$ \\ \cline{2-2} 
 & $\{-2, -2, -1, -1\}$ \\ \hline
14 & $\{-3, -2, -1, -1\}$ \\ \hline
\end{tabular}
\caption{Possible spectra for stable rank-$2$ reflexive sheaves with $c_1=0$ and $c_2=4$. The spectrum marked with an asterisk is not realized by a stable rank-$2$ reflexive sheaf, cf. Remark \ref{rem: badspectra}.
} 
\label{tab: spec c1=0}
\end{table}

\begin{table}[H] \def\arraystretch{1.5}
\begin{tabular}{|c|c|}
\hline
$c_3$ & spectra \\ \hline\hline
\multirow{2}{*}{{0}} & $\{-1, -1, 0, 0\}$ \\ \cline{2-2} 
 & $\{-2, -1, 0, 1\}$ \\ \hline
\multirow{2}{*}{2} & $\{-1, -1, -1, 0\}$ \\ \cline{2-2} 
 & $\{-2, -1, 0, 0\}$ \\ \hline
\multirow{2}{*}{4} & $\{-1, -1, -1, -1\}$ \\ \cline{2-2} 
 & $\{-2, -1, -1, 0\}$ \\ \hline
\multirow{2}{*}{6} & $\{-2, -1, -1, -1\}$ \\ \cline{2-2} 
 & $\{-2, -2, -1, 0\}$ \\ \hline
\end{tabular}\quad \quad
\begin{tabular}{|c|c|}\hline
$c_3$ & spectra \\ \hline\hline
\multirow{2}{*}{8} & $\{-2, -2, -1, -1\}$ \\ \cline{2-2} 
 & $\{-3, -2, -1, 0\}$ \\ \hline
\multirow{2}{*}{10} & $\{-2, -2, -2, -1\}$ \\ \cline{2-2} 
 & ${\{-3, -2, -1, -1\}}$ \\ \hline
12 & $\{-3, -2, -2, -1\}$ \\ \hline
14 & $*\{-3, -3, -2, -1\}$ \\ \hline
16 & $\{-4, -3, -2, -1\}$ \\ \hline
\end{tabular}
\caption{Possible spectra for stable rank-$2$ reflexive sheaves with $c_1=-1$ and $c_2=4$. The spectrum marked with an asterisk is not realized by a stable rank-$2$ reflexive sheaf, cf. Remark \ref{rem: badspectra}.
} 
\label{tab: spec c1=-1}
\end{table}

\begin{remark}\label{rem: badspectra}
There is no stable rank-two reflexive sheaf whose spectrum is either $\{-2, -2, -1, 0\}$, when $c_1 = 0$, or $\{-3, -3, -2, -1\}$, when $c_1=-1$. The first claim follows from \cite[Example 5.1.3]{H4}. For the second, it follows from \cite[Theorem A]{MR} that $\calr(-1, 4, 14)$ is empty, cf. also \cite[Proposition 2]{Ch3} and \cite[Example 5.1.4]{H4}. Both cases can be proved independently by studying Serre's correspondence and the constraints in cohomology discussed below.
\end{remark}

From Tables \ref{tab: spec c1=0} and \ref{tab: spec c1=-1} and the properties of the spectrum and the Hilbert Polynomials discussed above, we derive the following immediate results.

\begin{lemma} \label{lem: h2}
If $F$ is a stable rank-$2$ reflexive sheaf with $c_2=4$, then $h^2(F(2))=0$ and $h^2(F(1))\le1$ with equality only if $c_1=-1$ and $c_3=16$.
\end{lemma}
\begin{proof}
This follows from property \ref{h2spec} and checking in Tables \ref{tab: spec c1=0} and \ref{tab: spec c1=-1} that $-5$ does not appear in any spectrum, and $-4$ only occurs if $c_1=-1$ and $c_3=16$.
\end{proof}

\begin{lemma} \label{lem: h0}
If $F$ is a stable rank-$2$ reflexive sheaf with $c_2=4$, then $h^0(F(2))> 0$. Moreover, if either when $c_1=0$ and $c_3\ge10$ or $c_1=-1$ and $c_3\ge12$, then $h^0(F(1))>0$.
\end{lemma}
\begin{proof}
Stability implies that $h^3(F(k))=0$ for $k\ge-4$; then, from the formulas in Lemma \ref{lem: HRR} and Lemma \ref{lem: h2}, we get
\begin{equation}\label{eq: h0F2}
\chi(F(2)) = h^0(F(2)) - h^1(F(2)) = \begin{cases} \displaystyle
\frac{c_3 }{2} +4, & \text{if } c_1 = 0; \\[10pt] \displaystyle
\frac{c_3 }{2}, & \text{if } c_1 = -1.
\end{cases} 
\end{equation}
If either $c_1=0$, or $c_1=-1$ and $c_3>0$, it immediately follows that $h^0(F(2))>0$. If $c_1=-1$ and $c_3=0$, this was proved in \cite[Lemma 1]{BM}. Their argument is as follows. Assuming by contradiction that $h^0(F(2)) = 0$, one would get $F$ is $3$-regular and a general global section of $F(3)$ vanishes along a smooth curve of degree $10$ and genus $6$. Moreover, $C$ is canonical, i.e., $\omega_C = \OO_C(1)$. However, every such curve is contained in a quartic surface (cf. \cite[p.58]{GP-genre}), which is absurd since $h^0(F(2)) = h^0(\sI_C(4))$.

The formulas in Lemma \ref{lem: HRR} also yield
\begin{equation}\label{eq: h0F1}
\chi(F(1)) = h^0(F(1)) - h^1(F(1)) + h^2(F(1)) = \begin{cases} \displaystyle
\frac{c_3}{2} - 4, & \text{if } c_1 = 0; \\[10pt] \displaystyle
\frac{c_3}{2} -5, & \text{if } c_1 = -1.
\end{cases} 
\end{equation}
When either $c_1=0$ and $c_3\ge10$, or $c_1=-1$ and $c_3\ge12$, we get $h^0(F(1))>0$ as desired. 
\end{proof}


\section{Serre's construction} \label{sec: serre}

Let $F$ be a rank-$2$ reflexive sheaf and let $\sigma \in H^0(F(k))$, for some $k\in\Z$, without zeros in codimension one. Then $\sigma$ induces an exact sequence
\begin{equation}\label{eq: twist}
0 \longrightarrow \op3 \longrightarrow F(k) \overset{\sigma^\vee}{\longrightarrow} \sI_C(2k+c_1) \longrightarrow 0
\end{equation}
where $C$ is the vanishing locus of $\sigma$; this is a curve, i.e., a locally Cohen--Macaulay scheme of pure dimension one. Note that the sequence above corresponds to a nontrivial element of $\xi \in \Ext^1(\sI_C(2k+c_1), \op3) \simeq H^0(\omega_C(4-2k-c_1))$ and that the sheaf $F$ is reflexive if and only if $\xi$, as a section of $\omega_C(4-2k-c_1)$, vanishes in dimension 0 at most.

On the other hand, given a pair $(C, \xi)$ consisting of a curve and a section in $H^0(\omega_C(4-2k-c_1))$ vanishing only in dimension $0$, one produces a pair $(F,\sigma)$ consisting of a rank-$2$ reflexive sheaf and a global section $\sigma \in H^0(F(k))$. This is the content of the Serre correspondence; for details, we refer to \cite[\S 4]{H2}. The sheaf $F$ is stable if and only if $H^0(\sI_C(k+c_1))=0$.

When $c_2(F)=4$, Lemma \ref{lem: h0} implies that it is enough to consider $k=1,2$. It may be the case that $h^0(F(1))>0$, and $F(1)$ corresponds to a curve $C$ satisfying: 
\begin{equation}\label{eq: twist1-deg+g}
\deg(C) = c_1 + 5 \quad {\rm and} \quad p_a(C)= \dfrac{c_3}{2}+\dfrac{1}{2}(c_1+5) (c_1-2) +1 .
\end{equation}
Due to stability, $h^0(F) = h^0(\sI_C(2+c_1))=0$, i.e., $C$ is not contained in a surface of degree $c_1+1$ or less. 

When $h^0(F(1))=0$, the sheaf $F(2)$ corresponds to a curve $C$ satisfying: 
\begin{equation}\label{eq: twist2-deg+g}
\deg(C) = 2c_1 + 8 \quad {\rm and} \quad p_a(C)= \dfrac{c_3}{2}+(c_1+4)c_1+1.
\end{equation}
Moreover, $C$ is not contained in a surface of degree $c_1+2$ or less.

Also important is to consider a relative version of the Serre correspondence, i.e., the correspondence between families of pairs. This is well-known to specialists, though rarely spelled out in articles. Let us explain this in some detail.

Given a flat family of curves $\calc $ regarded as a subvariety of the appropriate Hilbert Scheme $\hilb^{d,g}(\p3)$ of curves of degree $d$ and arithmetic genus $g$, let $\mathbf{C}\subset\calc \times\p3$ be the universal curve and $\sI_{\mathbf{C}}$ be the universal ideal sheaf on $\calc \times\p3$. Let $\pi_1$ and $\pi_2$ be the projections of $\calc \times\p3$ onto its first and second factors, respectively. If we assume that $\calc $ is integral and that the function $C\mapsto h^0(\omega_C(4-2k-c_1))$ is constant on $\calc $, then the sheaf of relative $\Ext$ groups
\[
\mathcal{E} \coloneqq \inext^1_{\pi_1}(\sI_{\mathbf{C}}(k+c_1), \pi_2^*\op3(-k))
\]
is a vector bundle over $\calc $ parameterizing pairs $(C,\xi)$. Indeed, the complex $\mathcal{L}^{\bullet}$ in \cite[Corollary 1.2]{lange} has all the properties needed to prove a version of Grauert's Theorem \cite[III Corollary 12.9]{HART-AG} for the variation of $\Ext$. In the analytic category, this was done by B\u{a}nic\u{a}, Putinar, and Schumacher in \cite[Satz 3, (ii)]{BPS}.

Moreover, for $k >0$ and $c_1\in\{0,1\}$, we have $\Hom(\sI_C(k+c_1), \op3(-k)) = 0$ for every $C$. By \cite[Corollary 4.5]{lange}, there exists a universal extension on $\PE\times \p3$: 
\[
0 \longrightarrow p_2^*\op3(-k)\otimes p_1^*\OO_{\PE}(1) \longrightarrow \mathbf{F} \longrightarrow (\rho\times 1)^*\sI_{\mathbf{C}}\otimes p_2^*\op3(k+c_1) \longrightarrow 0
\]
where $p_i$ are the canonical projections of $\PE \times \p3$ and $\rho \colon \PE \to \calc $ is the structural morphism. Then $\mathbf{F}$ is a flat family of torsion-free sheaves. We are interested in the following open subset of $\PE$: 
\[
 \mathcal{U} \coloneqq \big\{(C,\xi) \in \PE \mid {\rm either}~ \dim(\xi)_0 = 0 ~{\rm or}~ (\xi)_0=\emptyset \big\} , 
\]
where $(\xi)_0$ denotes the zero locus of $\xi$ as a section of $\omega_C(4-2k-c_1)$. Hence, $\mathbf{F}_{\mathcal{U}}$ is a flat family of reflexive sheaves. If, moreover, $h^0(\sI_C(k+c_1)) = 0$ for every $C\in \calc $, then $\mathbf{F}_{\mathcal{U}}$ is a family of stable reflexive sheaves. We obtain a morphism 
\[
\Psi\colon \mathcal{U} \longrightarrow \calr(c_1,c_2,c_3),
\]
where $c_2$ and $c_3$ depend on $d$, $g$ and $k$, as in \cite[Theorem 4.1]{H2}. The fiber over $F$ can be identified with an open subset of $\mathbb{P}\big(H^0(F(k))\big)$. The closure of the image $\calf \coloneqq \overline{\Psi(\mathcal{U})}$ is an integral scheme parameterizing sheaves that correspond to curves in $\calc $. To summarize this discussion, we state the following proposition.

\begin{proposition}\label{prop: Serre-families}
Fix $k\in \Z_{>0}$ and $c_1\in \{-1,0\}$. Let $\calc \in \hilb^{d,g}(\p3)$ be an integral family of curves of degree $d$ and arithmetic genus $g$ such that $h^0(\omega_C(4-2k-c_1))$ and $h^0(\sI_C(2k+c_1))$ are constant on $\calc $, and $h^0(\sI_C(k+c_1)) = 0$ for every $C\in \calc $. Then Serre correspondence gives an irreducible family of sheaves $\mathcal{F} \subset \calr(c_1,c_2,c_3)$. Moreover, the dimensions are related by the following formula: 
\begin{equation}\label{eq: dimform-red}
\dim \calf + h^0(F(k)) = \dim\calc + h^0(\omega_C(4-2k-c_1)). 
\end{equation}
\end{proposition}

\begin{remark}\label{rem: dimformula}
The formula in display \eqref{eq: dimform-red} relates the dimensions of a family of curves to the corresponding family of sheaves. We refer to \cite{Kleppe} for a deformation theoretic approach to Serre correspondence. In particular, for a formula similar to \eqref{eq: dimform-red} regarding Zariski tangent spaces to the respective moduli spaces, cf. \cite[Theorem 2.1]{Kleppe}.
\end{remark}

Finally, we also implemented the Serre correspondence in the computer algebra system Macaulay2 \cite{M2}. In the ancillary file \verb|refshsetup.m2|, we have the function \verb|Serre(C, k)|, which computes a random extension of $\sI_C$ by $\op3(-k)$, if it exists. This computational approach will be used to provide explicit examples throughout the text. 


\subsection{Lemmas about curves}
We will narrow down the types of curves relevant to our work. As usual, by a curve, we mean a Cohen-Macaulay subscheme of $\p3$ of pure dimension one. One tool we will need to describe these curves is Liaison (or Linkage) Theory; our primary reference is \cite[III]{MDP}.

Let $C\subset \p3$ be a curve and $X$ be the complete intersection of two surfaces of degrees $s$ and $t$ containing $C$. Applying $\inhom_{\op3}(-, \calo_X)$ to the canonical map $\op3 \twoheadrightarrow \calo_C$, we get that 
$\inhom_{\op3}(\calo_C, \calo_X) \subset \calo_X$ is the ideal sheaf $\sI_{\Gamma/X}$ of a curve $\Gamma\subset X$. Conversely $\sI_{C/X} = \inhom_{\op3}(\calo_\Gamma, \calo_X)$. We say that $C$ and $\Gamma$ are \emph{linked} by $X$. These curves must satisfy the following properties: 
\begin{enumerate}[label={P.\arabic*}, ref=P.\arabic*]
 \item \label{P1} $\deg( C) + \deg( \Gamma) = st$ and $p_a(C)-p_a(\Gamma) = (\deg( C) - \deg( \Gamma))\frac{1}{2}(s+t -4)$;
 \item \label{P2} For any $n\in \Z$, we have $h^1(\sI_C(n)) = h^1(\sI_\Gamma(s+t-n-4)$;
 \item \label{P3} For any $n\in \Z$, we have $h^0(\sI_C(n)) = h^0(\sI_X(n))+ h^1(\OO_\Gamma(s+t-n-4))$;
 \item \label{P4} $M_C \simeq M_{\Gamma}^*(4-s-t) \coloneqq \Ext^4_R(M_{\Gamma}, R(-s-t))$, where $R \coloneqq H^0_*(\op3) $. Here, the actual module structure is being considered. 
\end{enumerate}
For proof, cf. \cite[III, Proposition 1.2]{MDP}.

A relative version of liaison is also available. Let $\calh_{\gamma, \rho} \subset \hilb^{d,g}(\p3)$ denote the Hilbert Scheme of curves with fixed postulation character $\gamma$ and Rao function $\rho$. If $C\in \calh_{\gamma, \rho}$ is linked to a curve $C\in \calh_{\gamma', \rho'}$ by surfaces of degree $s$ and $t$ then 
\begin{equation}\label{eq: dimform-liaison}
 \dim_C\calh_{\gamma, \rho} + h^0(\sI_C(s)) + h^0(\sI_C(t)) = \dim_{C'}\calh_{\gamma', \rho'} + h^0(\sI_{C'}(s)) + h^0(\sI_{C'}(t))
\end{equation}
cf. \cite[VII, Corollaire 3.8]{MDP} and also \cite[Theorem 7.1]{Kleppe-liaison}. An analogous dimension formula holds for linked families of curves. 

From \eqref{eq: twist1-deg+g} and \eqref{eq: twist2-deg+g}, we only need to deal with curves of degrees $4, 5, 6$, and $8$. Curves of degree $4$ were described by Nollet and Schlesinger in \cite{NS-deg4}. However, not much is known in the literature about curves of higher degrees. We start by classifying some quintic curves appearing in our setting.

{Recall that a curve $C$ is called extremal if it is not planar and $h^0(\sI_C(2)) \geq 2$. In this case, $C$ is either a twisted cubic, an elliptic quartic, or it contains a planar subcurve of degree one less than its degree. The terminology comes from the Rao function: extremal curves attain the maximum possible values for the Rao function. Analogously, subextremal curves have the maximal Rao functions among non-extremal curves. We refer to \cite[\S 9]{HS} for properties of extremal and subextremal curves.}

\begin{lemma}\label{lem: curv5quad}
Let $C$ be a non-planar curve of degree $5$ on $\p3$. If $h^0(\sI_C(2))>0$, then one of the following holds: 
\begin{enumerate}
 \item $h^0(\sI_C(2)) = 2$ and $C$ is extremal, there exists $C' \subset C$ a plane quartic;
 \item $h^0(\sI_C(2)) = 1$ and $C$ is either 
 \begin{enumerate}
 \item a divisor of type $(0, 5)$ on a smooth quadric. In particular, 
 \[
 h^1(\sI_C(t)) = \begin{cases}4, & t\in\{0, 3\}, \\ 6, & t\in \{1, 2\}, \\ 0, & {\text otherwise.} \end{cases}
 \]
 \item a divisor of type $(1, 4)$ on a smooth quadric. In particular, 
 \[
 h^1(\sI_C(t)) = \begin{cases}2, & t\in\{1, 2\}, \\ 0, & {\text otherwise.} \end{cases}
 \] 
 \item a subextremal curve. In particular, 
 \[
 h^1(\sI_C(t)) = \begin{cases}2-p_a(C), & t\in\{1, 2\}, \\ 
 \max\{1-p_a(C)+t, 0\}, & t\leq 0, \\
 \max\{4-p_a(C)-t, 0\}, & t\geq 3. 
 \end{cases}
 \]
 \item a negative genus curve on a double plane that is not subextremal. 
 \end{enumerate}
\end{enumerate}
\end{lemma}

\begin{proof}
Assume that $r = h^0(\sI_C(2)) \geq 2$, i.e., $C$ is extremal. By the previous paragraph, $C$ must contain a planar quartic subcurve since $\deg(C) = 5$. Now assume that $h^0(\sI_C(2))= 1$ and let $Q$ be the quadric containing $C$.

If $Q$ is smooth, then $C$ is a divisor of bidegree $(a, 5-a)$ on $Q\simeq \p1 \times \p1$, and $0\leq a\leq 2$. It follows that either: 
\begin{itemize}
 \item $a = 0$ and $C$ is linearly equivalent on $Q$ to a disjoint union of $5$ lines. In particular, one may compute: $h^1(\sI_C) = 4$, $h^1(\sI_C(1)) = 6$, $h^1(\sI_C(2)) = 6$, $h^1(\sI_C(3)) = 4$, and $h^1(\sI_C(l)) = 0$ for $l<0$ or $l>3$.
 \item $a=1$ and $C$ can be linked to a disjoint union of $3$ lines by the complete intersection of $Q$ and a quartic surface containing $C$. Hence, one may compute $h^1(\sI_C(1)) = h^1(\sI_C(2)) = 2$ and $h^1(\sI_C(l))= 0$ for $l \not \in \{1, 2\}$.
 \item $a=2$ and $C$ is linked to a line by the intersection of $Q$ with a cubic surface. Hence, $C$ is ACM of genus $2$, which is subextremal.
\end{itemize}

If $Q$ is a quadratic cone, then $Q\simeq \mathbb{P}_{(1, 1, 2)}$ embedded by $\OO_{\mathbb{P}_{(1, 1, 2)}}(2)$, {see \cite[p.128]{Harris}}. The weighted degree map yields $\Pic(Q) \simeq 2\Z \subset Cl(Q)\simeq \Z$, cf. \cite[\S4]{CLS-toric}. Then, any Weil divisor on $Q$ of even degree is Cartier. Moreover, any Cartier divisor of $Q$ comes from the embedding into $\p3$. This is to say that any curve in $Q$ of even degree is a complete intersection. Furthermore, any odd-degree curve in $Q$ is linked to a general line through the cone's vertex, so it is ACM. In our case, $C$ has degree $5$, so it is an ACM curve of genus $2$. 

If $Q = H_1 \cup H_2$ is the union of two distinct planes, then taking $C\cap H_1$ yields
\[
0 \longrightarrow \sI_Y(-1) \overset{h_1}{\longrightarrow} \sI_C \longrightarrow \sI_{Z/H_1}(-P) \longrightarrow 0 
\]
where $P$ is the one-dimensional component of $C\cap H_1$, and $Z$ is a zero-dimensional scheme whose support is contained in $L \coloneqq H_1 \cap H_2$. Also, the residual curve $Y$ is contained in $H_2$. Note that $\deg(C) = \deg(Y) + \deg(P)$, and, up to exchanging $H_1$ and $H_2$, we may assume that $\deg(P) > \deg(Y)$. Since $C$ is neither planar nor extremal, $\deg (P) = 3$ and $C$ is a subextremal curve, cf. \cite[Proposition 9.9]{HS}. Moreover, $p_a(C) = 2 - h^0(\OO_Z)$.

If $Q = V(h^2)$ is a double plane, then let $H= V(h)$, let $P$ be the maximal one-dimensional subscheme of $C\cap H$, and let $Y$ be the residual scheme to $C\cap H$ in $C$. Thus,
\[
0 \longrightarrow \sI_Y(-1) \overset{h}{\longrightarrow} \sI_C \longrightarrow \sI_{Z/H}(-p) \longrightarrow 0 
\]
where $Z$ is a $0$-dimensional subscheme of $H$ and $p = \deg P$. Due to \cite[Proposition 2.1]{HS}, $Z\subset Y\subset P \subset H$ and $\deg Y = 5-p$ and $p_a(C)=2-h^0(\OO _Z)$. Since $C$ is not extremal, $p \leq 3$. On the other hand, $5-p \leq p$ because $Y\subset P$, hence $p=3$. If $Z$ is contained in a line, then $C$ is subextremal; this is always the case when $p_a(C)\geq 0$.

\end{proof}

For curves of degree $6$, we have the following result. 

\begin{lemma}\label{lem: curv6cub}
 Let $C$ be a degree-$6$ curve, and let $g$ denote its arithmetic genus. If $h^0(\sI_C(2))=0$ and $h^0(\sI_C(3))\geq 2$ then: 
 \begin{enumerate}
 \item $C$ is obtained by direct link from a curve of degree $3$ and genus $g-3$.
 \item $C$ is given by 
 \[
 0 \longrightarrow \sI_Y(-1) \overset{h}{\longrightarrow} \sI_C \longrightarrow \sI_{Z/H}(-d) \longrightarrow 0 
 \]
 where $H$ is a hyperplane, $Z\subset H$ has dimension $0$, and either \begin{enumerate}
 \item $d=2$, $Y$ is a complete intersection. In addition, $h^0(\sI_C(3)) = 2$ and $p_a(C) = 4-h^0(\OO_Z) \leq 1$.
 \item $d=3$, $Y$ is a (possibly degenerate) twisted cubic. Then, $h^0(\sI_C(3)) = 3$ and $p_a(C) = 3- h^0(\OO_Z)\leq 2$.
 \item $d=4$, $Y$ is of degree $2$ and genus $k<0$. Moreover, $p_a(C) = 4 + k -h^0(\OO_Z)$, and $h^0(\sI_C(3)) = 4$ if $k=-1$, and $h^0(\sI_C(3)) = 3$ if $k\leq-2$.
 \end{enumerate}
 \item $C$ has a subcurve of degree $5$ lying on a quadric. Moreover, $h^0(\sI_C(3)) = 2$.
 \end{enumerate}
\end{lemma}

\begin{proof}
Let $f_1, f_2 \in H^0(\sI_C(3))$ be general elements. If $\gcd(f_1, f_2) = 1$ then $V(f_1, f_2)$ link $C$ to a curve $Y$ of degree $3$ and {genus $g-3$}.

Now assume $\gcd(f_1, f_2) = f$ has degree $1$, i.e., $f$ divides every element of $H^0(\sI_C(3))$. Taking the intersection $C\cap H$, we get a curve of degree $d\leq 5$ and a zero-dimensional scheme $Z$, satisfying
\[
0 \longrightarrow \sI_Y(-1) \overset{h}{\longrightarrow} \sI_C \longrightarrow \sI_{Z/H}(-d) \longrightarrow 0, 
\]
where $Y$ is the residual curve. Note that $H^0(\sI_Y(2))\cdot h = H^0(\sI_C(3))$ thus $h^0(\sI_Y(2)) \geq 2$ and $Y$ must be an extremal curve of degree $6-d$. Note that $f_1 = hq_1$ and $f_2 = hq_2$ where $q_1, q_2\in H^0(\sI_Y(2))$ are coprime. Then $6-d \leq 4$, i.e., $d\geq 2$. On the other hand, $d\leq 4$ since $d=5$ would imply $h^0(\sI_C(2))>0$. 

For $d=2$ we have that $Y$ is a complete intersection; in particular $h^0(\sI_C(3))= 2$. From the sequence above, we get that $h^0(\sI_{Z/H}(1)) \leq h^1(\sI_Y(2)) = 0$, hence $h^0(\OO_Z) \geq 3$. Taking Euler characteristics of the same sequence, we obtain 
\[
p_a(C) = \chi(\sI_C) = \chi(\sI_Y(-1)) + \chi(\sI_{Z/H}(-2)) = 4 - h^0(\OO_Z) \leq 1
\]

For $d=3$ then $V(q_1, q_2)$ link $Y$ to a line, hence $Y$ is a (possibly degenerate) twisted cubic; hence $h^0(\sI_C(3))= 3$. From the exact sequence above, $h^0(\sI_{Z/H})=0$, then $h^0(\OO_Z)\geq 1$, and computing the genus we get $p_a(C) = 3-h^0(\OO_Z) \leq 2$

If $d=4$, then $Y$ has degree $2$, but $h^0(\sI_Y(1)) = h^0(\sI_C(2)) = 0$. Hence, $Y$ has genus $k<0$, and it is given, up to a linear change of coordinates, by an ideal of the form $I_Y = (x^2, xy, y^2, xf-yg)$, where $\deg f = \deg g = -k$, cf. \cite{N-deg3}. Then $h^0(\sI_C(3))=h^0(\sI_Y(2)) = 3$ for $k\leq -2$ and $h^0(\sI_C(3))= 4$ for $k=-1$. Moreover, the genus of $C$ is $p_a(C) = 4+k-h^0(\OO_Z) \leq 4+k$.

Now assume $\gcd(f_1, f_2) = q$ of degree $2$, i.e., $q$ divides every element of $H^0(\sI_C(3))$. Then, taking the intersection $C\cap Q$, we may argue as in the previous case. We have an exact sequence:
\[
0 \longrightarrow \sI_Y(-2) \overset{q}{\longrightarrow} \sI_C \longrightarrow \sI_{Z/Q}(-C') \longrightarrow 0, 
\]
where $C'$ is the one-dimensional part of $C\cap Q$ and $Y$ is the residual curve. Note that $h^0(\sI_Y(1)) = h^0(\sI_C(3)) \geq 2$. Since $h^0(\sI_C(2))=0$, we have that $Y$ is a line and $C'$ is a degree $5$ curve on a quadric surface. Moreover, $h^0(\sI_C(3)) = 2$.
\end{proof}

Next, we use our knowledge of curves to provide upper bounds for the global sections of the corresponding sheaves.

\begin{lemma}\label{lem: h0F1>0}
Let $F$ be a stable rank-$2$ reflexive sheaf with $c_2(F)=4$ such that $h^0(F(1))>0$.
\begin{itemize}
\item[(i)] If $c_1(F)=0$, then $h^0(F(1))\leq3$ and equality holds only when $-3$ figures in the spectrum; hence $c_3(F) \geq 12$.
\item[(ii)] If $c_1(F)=-1$, then $h^0(F(1))\le2$, and equality only occurs when $c_3(F)=16$. 
\end{itemize}
\end{lemma}
\begin{proof}
If $c_1(F)=0$, then the exact sequence in display 
\eqref{eq: twist} for $k=1$ yields $h^0(F(1))=1+h^0(\sI_C(2))$, so if $h^0(F(1))>1$, then $C$ is contained in a quadric, while stability implies that $C$ is not planar; in addition, $\deg(C)=5$ and $p_a(C)=c_3/2-4$ by \eqref{eq: twist1-deg+g}.

Suppose that $h^0(F(1))\geq 3$. Then $C$ is an extremal curve of degree $5$. By Lemma \ref{lem: curv5quad}, $h^0(F(1)) = 3$. Moreover, we have an epimorphism $\sI_C \twoheadrightarrow \sI_{Z/H}(-4)$, where $H$ is a plane and $Z$ is a $0$-dimensional subscheme of $H$. Composing this epimorphism with $F(-1) \twoheadrightarrow \sI_C$ we see that $H$ is an unstable plane of order $3$ for $F$. Therefore $h^2(F)>0$ and $-3$ occurs in the spectrum, and this only occurs when $c_3(F)\geq 12$ (see Table \ref{tab: spec c1=0}).

When $c_1(F)=-1$, then the exact sequence in display \eqref{eq: twist} for $k=1$ yields $h^0(F(1))=1+h^0(\sI_C(1))$, so $h^0(F(1))>1$ if and only if $C$ is a planar curve; note that $\deg(C)=4$. Since $h^0(\sI_C(1))>1$ if and only if $C$ is a line, we conclude that $h^0(F(1))\le2$. Moreover, the equality can only occur when $p_a(C)=(\deg(C)-2)(\deg(C)-1)/2=3$; from the rightmost formula in display \eqref{eq: twist1-deg+g} we have $3=c_3/2-5$, thus $c_3=16$.
\end{proof}


\section{Sheaves corresponding to extremal curves}\label{sec: extremal}

Let $C$ be a degree-$d$ extremal curve, i.e., $C$ is non-planar and $h^0(\sI_C(2)) \geq 2$. This section describes the reflexive sheaves that can be built from $C$ via Serre correspondence. First, we need to describe $C$ better. If $C$ is not ACM, then it has a planar subcurve $Y\subset C$ of degree $d - 1$ and a residual line $L$ fitting in the following exact sequence
\begin{equation}\label{seq: extremal}
 0 \longrightarrow \sI_L(-1) \overset{h}{\longrightarrow} \sI_C \longrightarrow \sI_{Z/H}(1-d) \longrightarrow 0, 
\end{equation}
where $H = V(h)$ is a hyperplane, and $Z\subset H\cap L$ is a zero-dimensional subscheme. Note that $ h^0(\OO_Z) \geq 1$, since $C$ is not ACM. It also follows from this sequence that $p_a(C) = \frac{1}{2}(d-2)(d-3) - h^0(\OO_Z)$.
The main result of this section is the following: 

\begin{proposition}\label{prop: extremal-ext}
Let $F$ be a normalized reflexive sheaf, $c_1 = c_1(F) \in \{0, -1\}$, such that a global section of $F(k)$ vanishes precisely at a non-ACM extremal curve $C$ for some $k \in \Z$. Let $\xi \in H^0(\omega_C(4-2k-c_1))$ be the extension class associated with $F$. If $h^0(F) = 0$ then $k=1$ and 
\[
\extd^2(F, F) \geq h^1(F(d-6-c_1))
\]
with equality if $d\leq 5+c_1$ and $\xi|_Z \neq 0$. In particular, if $\xi|_Z \neq 0$ then
\begin{equation}\label{eq: ext2-extremal}
 \extd^2(F, F) = \begin{cases}
 1, & c_1 = 0, \text{ and } d = 5 \\
 h^0(\OO_Z) -1, & c_1 = -1, \text{ and } d = 4 
\end{cases}.
\end{equation}
If $\xi|_Z = 0$ then $c_1 = -1$, $ h^0(\OO_Z) = 1$ and for $d=4$ we have
\begin{equation}\label{eq: ext2-extremal1}
 \extd^2(F, F) \leq 1.
\end{equation}
\end{proposition}

We start by bounding $k$. Let $F$ be a normalized reflexive sheaf and $k$ be an integer such that $C$ is the vanishing locus of a global section of $F(k)$. Then $F(k)$ corresponds to $\xi \in H^0(\omega_C(4-2k-c_1))$ with isolated zeros. Dualizing and twisting the sequence \eqref{seq: extremal}, we get {the exact sequence}
\[
0 \longrightarrow \OO_Y(d-2k-c_1) \longrightarrow \omega_C(4-2k-c_1) \longrightarrow \OO_L(3-2k-c_1- h^0(\OO_Z) ) \longrightarrow 0.
\]
Since $\xi$ has isolated zeros, $\xi|_L \not \equiv 0$ which implies $3-2k-c_1- h^0(\OO_Z) \geq 0$. On the other hand, if we suppose that $F$ is stable, then 
\[
2 \leq 2k \leq 3-c_1 - h^0(\OO_Z) .
\]
This leads to the following immediate lemma.

\begin{lemma}\label{lem: extrmal-Zbound}
Let $C$ be a non-ACM extremal curve, such that there exists $\xi \in H^0(\omega_C(4-2k-c_1))$ with isolated zeros. If $k\geq 1$ then $k=1$ and either: $c_1=0$ and $ h^0(\OO_Z) = 1$; or $c_1=-1$ and $ h^0(\OO_Z) \in \{1, 2\}$.
\end{lemma}

A reflexive sheaf $F$ as above fits in an exact sequence
\begin{equation}\label{seq: F-extremal}
 0 \longrightarrow \op3(-1) \longrightarrow F \longrightarrow\sI_C(1+c_1) \longrightarrow 0,
\end{equation}
{computing the long exact sequence of cohomlogy} we know that $H^1_*(F) \cong M_C(1+c_1)$ as $\kappa[x_0, x_1, x_2, x_3]$-modules. On the other hand, if follows from \eqref{seq: extremal} that the Rao module of $C$ is 
\[
M_C = \kappa[x_0, x_1, x_2, x_3]/(l_1, l_2, p, q)
\]
where $L = V(l_1, l_2)$, $Z = V(l_1, l_2, p)$, and $q$ is a degree $d-1$ polynomial not vanishing on $Z$. In particular, we can compute the dimensions
\begin{equation}\label{eq: coh-F-extremal}
 h^1(F(t)) = h^1(\sI_C(t +1+c_1) = \begin{cases}
 h^0(\OO_Z) , & -1-c_1\leq t \leq d-3-c_1 \\
 h^0(\OO_Z) -1, & t= -2-c_1, d-2-c_1 \\
 0, & \text{otherwise}
\end{cases},
\end{equation}
cf. \cite[p.16]{HS}. Note that, $h^1(F(t)) = 0$ for $t\leq -2$. 

\begin{remark}\label{rem: h-annihilates}
It follows from the description above that $h$ annihilates $H^1_*(F)$. Indeed, since $Z\subset H$, we have that $h = al_1+bl_2+cp$, where $c=0$ if $ h^0(\OO_Z) > 1$.
\end{remark}

\begin{proof}[Proof of Proposition \ref{prop: extremal-ext}] We have already settled that $k=1$ if $h^0(F) = 0$. Next, we compute $\extd^2(F, F)$. Consider the following commutative diagram, {with exact rows and columns}: 
\begin{equation}\label{eq: diag-extremal}
\begin{split}
\xymatrix@-2ex{ 
& 0 \ar[d] & 0 \ar[d] & \\
& \op3(-1)\ar[d]\ar@{=}[r] & \op3(-1)\ar[d] & & \\
0\ar[r] & E \ar[r]\ar[d] & F \ar[r] \ar[d] & \sI_{Z/H}(2+c_1-d) \ar@{=}[d] \ar[r] & 0 \\
0\ar[r] & \sI_L(c_1) \ar[r]\ar[d] & \sI_C(1+c_1) \ar[r]\ar[d] & \sI_{Z/H}(2+c_1-d) \ar[r] & 0 \\
& 0 & 0 & 
}
\end{split}
\end{equation}
where $E$ is another reflexive sheaf. Applying $\Hom(F, -)$ to the middle row we get {the exact sequence}
\[
\to \Ext^2(F, E) \longrightarrow \Ext^2(F, F) \longrightarrow \Ext^2(F, \sI_{Z/H}(2+c_1-d))\longrightarrow \Ext^3(F, E) \to .
\]
On the other hand, from the leftmost column, we have that $E$ has the resolution.
\[
0 \longrightarrow \op3(c_1-2) \longrightarrow \op3(c_1-1)^{\oplus 2} \oplus \op3(-1) \longrightarrow E \longrightarrow 0.
\]
Applying $\Hom(F, -)$ and Serre duality, we get $\Ext^2(F, E) = \Ext^3(F, E) = 0$ from $h^0(F) = 0$ and $h^1(F(t)) = 0$ for $t\leq -2$. Therefore, 
\[
\Ext^2(F, F) \cong \Ext^2(F, \sI_{Z/H}(2+c_1-d)).
\]
To compute $\Ext^2(F, \sI_{Z/H}(2+c_1-d))$, we apply $\Hom(F, -)$ to the exact sequence
\[
0 \longrightarrow \sI_{Z/H}(2+c_1-d) \longrightarrow\OO_H(2+c_1-d) \longrightarrow \OO_Z \longrightarrow 0
\]
to get {the exact sequence}
\[
\to \Ext^1(F, \OO_Z) \to \Ext^2(F, \sI_{Z/H}(2+c_1-d)) \to \Ext^2(F, \OO_H(2+c_1-d))\to 0. 
\]
Note that $\Ext^2(F, \OO_Z) = 0$, since $F$ is reflexive and $\dim Z = 0$.

We claim that $\Ext^1(F, \OO_Z) = 0$, unless $\xi|_Z = 0$; the proof will be given in Lemma~\ref{lem: vanish-extremal} below. Therefore, if $\xi|_Z \neq 0$ we have:
\[
\Ext^2(F, F) \cong \Ext^2(F, \OO_H(2+c_1-d)) \cong \Ext^1(\OO_H(6+c_1-d), F)^*,
\]
where the rightmost isomorphism is Serre duality. To compute the latter we apply $\Hom( - , F)$ to the exact sequence 
\[
0 \longrightarrow \op3(5+c_1-d) \overset{h}{\longrightarrow} \op3(6+c_1-d) \longrightarrow \OO_H(6+c_1-d) \longrightarrow 0
\]
Then, we have:
\[
H^0(F(d-5-c_1)) \to \Ext^1(\OO_H(6+c_1-d), F) \to H^1(F(d-6-c_1)) \overset{h}{\to } H^1(F(d-5-c_1)).
\]
On the other hand, $h$ annihilates the module $H^1_*(F)$, cf. Remark \ref{rem: h-annihilates}. Thus, 
\[
H^0(F(d-5-c_1)) \to \Ext^1(\OO_H(6+c_1-d), F) \to H^1(F(d-6-c_1)) \to 0, 
\]
and we have that $\extd(F, F) \geq h^1(F(d-6-c_1))$ with equality when $d\leq 5+c_1$, since $h^0(F) = 0$. Finally, to get \eqref{eq: ext2-extremal} we only need to use \eqref{eq: coh-F-extremal}.

If $\xi|_Z = 0$ we have, owing to Lemma \ref{lem: vanish-extremal}, that $ h^0(\OO_Z) = -c_1 = 1$ and $\Ext^1(F, \OO_Z) =1$. Nonetheless, it follows from the computations in the preceding paragraph that 
\[
\extd^2(F,F) \geq \extd^1(\OO_H(6+c_1-d), F) \geq h^1(F(d-6-c_1)).
\]
Restricting to the case $d=4$ we have that $\Ext^2(F, \OO_H(-3)) = 0$. Hence 
\[
\extd^2(F,F) = \extd^2(F,\sI_{Z/H}(-3)) \leq \extd^1(F,\OO_Z) = 1.
\]
\end{proof}

\begin{lemma}\label{lem: vanish-extremal}
Let $F$ be a normalized reflexive sheaf such that $F(1)$ corresponds to an extremal curve $C$ as in \eqref{seq: extremal}, with extension class $\xi \in H^0(\omega_C(2-c_1))$. If $\xi|_Z = 0$ then $h^0(\OO_Z) = -c_1 = 1$ and $\extd^1(F, \OO_Z) = 1$, otherwise $\extd^1(F, \OO_Z) = 0$.
\end{lemma}

\begin{proof}
The strategy is to compute the Ext groups explicitly from a free resolution for $F$, which is tied to the description of $C$. Up to a linear change of coordinates, we may assume $L= V(x_0, x_1)$, and we have four cases: 
\begin{enumerate}
 \item If $ h^0(\OO_Z) = 1$ and $L\not \subset H$, then we can choose $h = x_2$ and $Z = V(x_0, x_1, x_2)$;
 \item If $ h^0(\OO_Z) = 1$ and $L\subset H$, then we can choose $h = x_0$ and $Z = V(x_0, x_1, x_2)$;
 \item If $ h^0(\OO_Z) = 2$ and $Z$ is reduced, then we can choose $h = x_0$ and $Z = V(x_0, x_1, x_2x_3)$;
 \item If $ h^0(\OO_Z) = 2$ and $Z$ is not reduced, then we can choose $h = x_0$ and $Z = V(x_0, x_1, x_2^2)$.
\end{enumerate}

In the first case, $\sI_{Z/H}$ has the following resolution:
\begin{equation}\label{eq: res-Z-i}
 0 \to \op3(-3) \xrightarrow{\footnotesize \begin{bmatrix}
 x_0 \\ x_1 \\ x_2
 \end{bmatrix} } \op3(-2)^{\oplus 3} \xrightarrow{\footnotesize \begin{bmatrix}
 x_2 & 0 & -x_0 \\ 0 & x_2 & -x_1
 \end{bmatrix} } \op3(-1)^{\oplus 2} \longrightarrow \sI_{Z/H} \to 0 .
\end{equation}
This is the mapping cone associated with $\cdot x_2\colon \sI_Z(-1) \to \sI_Z$ and the Koszul resolution of $\sI_Z$. {For the construction of the mapping cone we refer to \cite[\S1.5]{Wei}}. With \eqref{eq: res-Z-i} and the Koszul resolution of $\sI_L$ in hand, we apply the Horseshoe Lemma to \eqref{seq: extremal} to obtain, up to vector bundle automorphisms, the following resolution of $\sI_C$: 
\begin{equation}\label{seq: res-C-extremal}
 0 \to \op3(-2-d) \xrightarrow{\footnotesize \begin{bmatrix}
 x_0 \\ x_1 \\ x_2 \\ x_3^{d-1}
 \end{bmatrix}} 
 \begin{matrix}
 \op3(-1-d)^{\oplus 3} \\\oplus \\ \op3(-3)
 \end{matrix}
 \overset{N}{\longrightarrow}
 \begin{matrix}
 \op3(-d)^{\oplus 2} \\ \oplus \\ \op3(-2)^{\oplus 2}
 \end{matrix}
 \longrightarrow \sI_C \to 0
\end{equation}
where, for some $f, g \in H^0(\op3(d-2))$, 
\[
N = \begin{bmatrix}
 x_2 & 0 & -x_0 & 0 \\ 
 0 & x_2 & -x_1 & 0 \\ 
 f\, x_1 -x_3^{d-1} & -f\, x_0 & 0 & x_0 \\
 g \, x_1 & -g\, x_0 - x_3^{d-1} & 0 & x_1
 \end{bmatrix}.
\]
Note that $(f, g)$ determines the extension class defining \eqref{seq: extremal}. Applying the Horseshoe Lemma to \eqref{seq: F-extremal} yields 
\begin{equation}\label{seq: res-F-extremal1}
0 \to \op3(c_1-1-d) \xrightarrow{\footnotesize \begin{bmatrix}
 x_0 \\ x_1 \\ x_2 \\ x_3^{d-1}
 \end{bmatrix}} 
 \begin{matrix}
 \op3(c_1-d)^{\oplus 3} \\\oplus \\ \op3(c_1-2)
 \end{matrix}
 \overset{A}{\longrightarrow}
 \begin{matrix}
 \op3(c_1+1-d)^{\oplus 2} \\ \oplus \\ \op3(c_1-1)^{\oplus 2} \\ \oplus \\ \op3(-1)
 \end{matrix}
 \to F\to 0
\end{equation}
where the matrix $A$ is the concatenation of $N$ with another row: 
\[
A = \begin{bmatrix}
 x_2 & 0 & -x_0 & 0 \\ 
 0 & x_2 & -x_1 & 0 \\ 
 f\, x_1 -x_3^{d-1} & -f\, x_0 & 0 & x_0 \\
 g \, x_1 & -g\, x_0 - x_3^{d-1} & 0 & x_1 \\
 f'x_1 & -f'x_0 & g'x_3^{d-1} & -g'x_2
 \end{bmatrix}, 
\]
for some $f'\in H^0(\op3(d-2-c_1)$ and $g' \in H^0(\op3(-c_1))$. We remark that the extension class $\xi \in \Ext^1(\sI_C(1+c_1),\op3(-1))$ defining $F$ is precisely the class of the bottom row of $A$, modulo the module generated by the rows of $N$. This can be seen directly by applying $\inhom(-, \op3(-1))$ to \eqref{seq: res-C-extremal}. Also note that $g'|_L \not \cong 0$, otherwise $L$ would be in the support of $\inext^1(F, \op3)$, hence $F$ would not be reflexive.

Next, we compute the sheaves $\inext^i(F, \OO_Z)$ by applying the functor $\inhom( - , \OO_Z)$ to \eqref{seq: res-F-extremal1}: 
\[
0 \longrightarrow \OO_Z^{\oplus 5} \xrightarrow{A^T|_Z} \OO_Z^{\oplus 4} \xrightarrow{\footnotesize \begin{bmatrix}
 0 & 0 & 0 & 1
\end{bmatrix}}
\OO_Z \longrightarrow 0
\]
We note that 
\[
A^T|_Z = \begin{bmatrix}
 0 & 0 & -1 & 0 & 0 \\ 
 0 & 0 & 0 & -1 & 0 \\ 
 0 & 0 & 0 & 0 & g'(0, 0, 0, 1) \\ 
 0 & 0 & 0 & 0 & 0 
 \end{bmatrix}
\]
If $c_1 =0$, then $g'$ is a constant. We have already pointed out that $F$ is reflexive only if $g' \neq 0$, in which case the complex above is exact, except on the right. Then we have $\inhom(F, \OO_Z) \cong \OO_Z^{\oplus 2}$ and $\inext^j(F, \OO_Z) = 0$ for $j\geq 1$. It follows from the local-to-global spectral sequence that 
\[
\Ext^1(F, \OO_Z) = 0.
\]
For $c_1 = -1$ we have that $g'$ is a degree-$1$ polynomial. Thus, for $g'$ generic, we get the vanishing of Ext groups as before. But for {$g'$ a constant multiple of $x_2$} we get $\inext^1(F, \OO_Z) \cong \OO_Z$. Hence 
\[
\extd^1(F, \OO_Z) = h^0(\inext^1(F, \OO_Z)) = 1.
\]

In the remaining cases, we repeat this argument \emph{mutatis mutandis}. For the second, $F$ has a resolution of the form \eqref{seq: res-F-extremal1} but the matrix $A$ is now
\[
A = \begin{bmatrix}
 x_2 & 0 & -x_0 & 0 \\ 
 x_1 & -x_0 & 0 & 0 \\ 
 -x_3^{d-1} & f\, x_2 & -f\, x_1 & x_0 \\
 0 & g \, x_2 - x_3^{d-1} & -g\, x_1 & x_1 \\
 0 & f'x_2 & g'x_3^{d-1}-f'x_1 & -g'x_2
 \end{bmatrix}.
\]
Hence, the proof is precisely as in the first case.

For the third and fourth cases, $ h^0(\OO_Z) =2$ and $c_1=-1$, and we have a resolution for $F$ of the form 
\begin{equation}\label{seq: res-F-extremal2}
0 \to \op3(-3-d) \overset{B}{\longrightarrow}
 \begin{matrix}
 \op3(-2 -d)^{\oplus 2} \\\oplus\\ \op3(-1 -d) \\\oplus \\ \op3(-3)
 \end{matrix}
 \overset{A}{\longrightarrow}
 \begin{matrix}
 \op3(-1-d) \\ \oplus \\ \op3(-d) \\ \oplus \\ \op3(-2)^{\oplus 2} \\ \oplus \\ \op3(-1)
 \end{matrix}
 \to F\to 0
\end{equation}

In the third case, we have 
\[
B = \begin{bmatrix} x_0 \\ x_1 \\ x_2x_3 \\ x_2^{d} + x_3^{d} \end{bmatrix}, \, 
A = \begin{bmatrix}
 x_1 & -x_0 & 0 & 0 \\
 x_2x_3 & 0 & -x_0 & 0 \\
 x_2^{d} + x_3^{d} & f\, x_2x_3 &-f\, x_1 & -x_0 \\
 0 & x_2^{d} + x_3^{d} + g\, x_2x_3 & -g\, x_1 & -x_1 \\
 0 & f'x_2x_3 & x_2^{d} + x_3^{d}-f'x_1 & -x_2x_3
\end{bmatrix}
\]
for some $f, g\in H^0(\op3(d-2))$, and $f'\in H^0(\op3(d-1))$. We have that $Z$ is composed of two points $\{ (0: 0: 0: 1), (0: 0: 1: 0)\}$. By symmetry, we will only analyze one of them, say $Z_1 = (0: 0: 0: 1)$. It follows that 
\[
B|_{Z_1} = \begin{bmatrix} 0 \\ 0 \\ 0 \\ 1 \end{bmatrix}, \, \text{ and } \, 
A|_{Z_1} = \begin{bmatrix}
 0 & 0 & 0 & 0 \\
 0 & 0 & 0 & 0 \\
 1 & 0 & 0& 0 \\
 0 & 1& 0 & 0 \\
 0 & 0& 1 & 0
\end{bmatrix}.
\]
Thus applying $\inhom(- , \OO_Z)$ to \eqref{seq: res-F-extremal2} and computing cohomology, we get $\inhom(F, \OO_Z) \cong \OO_Z^{\oplus 2}$ and $\inext^j(F, \OO_Z) = 0$ for $j\geq 1$. Therefore, $\Ext^1(F, \OO_Z) =0$.

Finally, we address the fourth case. We have a resolution \eqref{seq: res-F-extremal2} with maps given by $B = \begin{bmatrix} x_0 & x_1 & x_2^2 & x_2x_3^{d-1} + x_3^{d} \end{bmatrix}^T$, and 
\[
A = \begin{bmatrix}
 x_1 & -x_0 & 0 & 0 \\
 x_2^2 & 0 & -x_0 & 0 \\
 x_2x_3^{d-1} + x_3^{d} & f\, x_2^2 &-f\, x_1 & -x_0 \\
 0 & x_2x_3^{d-1} + x_3^{d} + g\, x_2^2 & -g\, x_1 & -x_1 \\
 0 & f'x_2^2 & x_2x_3^{d-1} + x_3^{d}-f'x_1 & -x_2^2
\end{bmatrix}
\]
for some $f, g\in H^0(\op3(d-2))$, and $f'\in H^0(\op3(d-1))$. Then, the restriction to $Z$ gives
\[
B = \begin{bmatrix} 0 \\ 0 \\ 0 \\ \overline{x_2} + 1 \end{bmatrix}, \, 
A = \begin{bmatrix}
 01 & 0 & 0 & 0 \\
 0 & 0 & 0 & 0 \\
 \overline{x_2} + 1 & 0 &0 & 0 \\
 0 & \overline{x_2} + 1 & 0 & 0 \\
 0 & 0 & \overline{x_2} + 1 & 0
\end{bmatrix}
\]
where $\overline{x_2}$ is the class of $x_2$ modulo $x_2^2$. In particular, $\overline{x_2} + 1$ is a unity in $\OO_Z$. Therefore, as in the previous case, we have $\Ext^1(F, \OO_Z) =0$.
\end{proof}


\section{A class of unobstructed reflexive sheaves}\label{sec: unobsheaves}

Given a reflexive sheaf $F$ on $\p3$ (of arbitrary rank), one can find a locally free sheaf $L=\bigoplus_i\op3(a_i)$ and an epimorphism $L\twoheadrightarrow F$; the kernel of this epimorphism, call it $E$, is necessarily locally free by the Auslander–Buchsbaum formula, or one can use the exact sequence
\begin{equation}\label{eq: ELF}
0 \longrightarrow E \overset{\varphi}{\longrightarrow} L \longrightarrow F \longrightarrow 0
\end{equation}
to show that $\inext^p(E, \op3)=0$ for $p>0$. 

In this section, we will present some sufficient conditions that allow us to compute $\extd^1(F, F)$ and guarantee that $F$ is unobstructed as a point in the corresponding moduli space of reflexive sheaves (that is, $\extd^2(F, F)=0$).

Given two locally free sheaves $E$ and $L$ such that $\rk(E)<\rk(L)$, set
\begin{equation}\label{def: M0}
\mathbb{M}_0 \coloneqq \big\{\, \varphi\in\Hom(E, L) \mid {\rm either}~ D(\varphi)=\emptyset ~{\rm or}~ \dim D(\varphi)=0 \,\}
\end{equation}
where $D(\varphi)\coloneqq \{\,x\in\p3 \mid \varphi(x) \textrm{ not injective} \,\}$ is the degeneration locus of the morphism $\varphi$. It follows that $F_\varphi\coloneqq \coker(\varphi)$ is reflexive for every $\varphi\in\mathbb{M}_0$.

Consider also the group
\begin{equation}\label{def: G}
\mathbb{G} \coloneqq \left(\Aut(E)\times\Aut(L)\right)\Big/\kappa^*\cdot{\rm id}.
\end{equation}
Note that $\mathbb{G}$ acts on $\mathbb{M}_0$ as follows:
\[ 
(g, h)\cdot\varphi\coloneqq h\circ \varphi\circ g^{-1}. 
\]
It is not difficult to see that if $\varphi'= (g, h)\cdot\varphi$, then $F_{\varphi'}\simeq F_\varphi$. If $F_\varphi$ is $\mu$-stable for every $\varphi\in\mathbb{M}_0$, we obtain {a natural modular morphism 
\[
\Psi \colon \mathbb{M}_0 \longrightarrow \calr(\mathbf{P})
\quad \textrm{given by} \quad 
\Psi(\varphi) = \big[F_\varphi\big] ~,
\]
}
\noindent where $\calr(\mathbf{P})$ denotes the moduli space of $\mu$-stable reflexive sheaves with Hilbert polynomial $\mathbf{P}(t)\coloneqq P_L(t)-P_E(t)$, and $[F]$ denotes the isomorphism class of a sheaf $F$ as a point in $\calr(\mathbf{P})$. 

\begin{lemma}\label{lem: lift}
If $\Hom(L, E)=\Ext^1(E, L)=0$, then an isomorphism $f\colon F_\varphi\overset{\sim}{\to} F_{\varphi'}$ lifts to unique isomorphisms $g\in\Aut(E)$ and $h\in\Aut(L)$ such that $\varphi'=h\circ \varphi\circ g^{-1}$. 
If, in addition, 
\[ 
\extd^1(F_\varphi, F_\varphi) = 
\homd(E, L) - \homd(E, E) - \homd(L, L) + 1 
\]
for every $\varphi\in\mathbb{M}_0$, then {$\mathbb{M}_0$ dominates an irreducible component of $\calr(\mathbf{P})$. }
\end{lemma}
\begin{proof}
Consider the following commutative diagram {with exact rows}
\[ 
\xymatrix{
0 \ar[r] & E \ar[r]^{\varphi}\ar@{..>}[d]^{g} & L \ar[r]\ar[dr]^{\tilde f}\ar@{..>}[d]^{h} & F_\varphi \ar[r]\ar[d]^{f} & 0 \\
0 \ar[r] & E \ar[r]^{\varphi'} & L \ar[r] & F_{\varphi'} \ar[r] & 0
}
\]
where $\tilde{f}$ is the obvious composition; the hypothesis implies that $\Hom(L, F_{\varphi'})\simeq\Hom(L, L)$, thus $\tilde{f}$ lift to a unique $h\in\Hom(L, L)$, and then one can find a unique $g\in\Hom(E, E)$ to complete the diagram. The hypothesis $\Hom(L, E)=0$ implies that $g$ and $h$ must be isomorphisms{; reverse the roles of $\varphi$ and $\varphi'$ to find their inverses. Therefore, the fibers of $\Psi$ are $\mathbb{G}$-orbits; more precisely, $\Psi^{-1}([F_\varphi])=\mathbb{G}\cdot\varphi$. A similar computation shows that the stabilizer of any $\varphi \in \mathbb{M}_0$ is trivial.}

Regarding the second claim, note that
\[ 
\dim \mathbb{M}_0 - \dim \mathbb{G} = \homd(E, L) - \homd(E, E) - \homd(L, L) + 1. 
\]
If this quantity, which is equal to the dimension of the {image of $\Psi$}, 
is equal to $\extd^1(F, F)$, then the conclusion follows.
\end{proof}


Next, we give a vanishing criterion for the obstruction space $\Ext^2(F, F)$ that will be useful later.

\begin{lemma}\label{lem: ext2=0}
If $F$ is a $k$-regular reflexive sheaf satisfying $h^0(F(k-3))=h^1(F(k-4))=0$, then $\Ext^2(F, F)=0$.
\end{lemma}
\begin{proof}
Since $F$ is $k$-regular, there is an epimorphism $\varepsilon\colon H^0(F(k))\otimes\op3(-k)\twoheadrightarrow F$; setting $E\coloneqq \ker(\varepsilon)$, one can check that $E$ is $(k+1)$-regular. Applying the functor $\Hom(F, -)$ to the exact sequence
\[ 
0 \longrightarrow E \longrightarrow H^0(F(k))\otimes\op3(-k) \longrightarrow F \longrightarrow 0 
\]
we obtain, setting $r\coloneqq h^0(F(k))$, {the exact sequence}
\[ 
\Ext^2(F, \op3(-k))^{\oplus r} \longrightarrow \Ext^2(F, F) \longrightarrow \Ext^3(F, E) \longrightarrow \Ext^3(F, \op3(-k))^{\oplus r}. 
\]
Serre duality yields $\Ext^p(F, \op3(-k)) \simeq H^{3-p}(F(k-4))^*$, and our hypothesis says that these vanish for $p=2, 3$. It thus follows that $\Ext^2(F, F) \simeq \Ext^3(F, E)$.

To conclude the proof, we apply the functor $\Hom(F, -)$ to the evaluation epimorphism
\[ 
H^0(E(k+1))\otimes\op3(-k-1) \twoheadrightarrow E 
\]
to see that
\[
\left(H^0(F(k-3))^{\oplus s}\right)^*\simeq \Ext^3(F, \op3(-k-1))^{\oplus s} \twoheadrightarrow \Ext^3(F, E) 
\]
where $s\coloneqq h^0(E(k+1))$; since $h^0(F(-k-3))=0$ by hypothesis, we conclude that $\Ext^2(F, F)=0$, as desired.
\end{proof}



\section{Examples of stable reflexive sheaves with \texorpdfstring{$c_1 = 0$}{c1 = 0} and \texorpdfstring{$c_3\leq 6$}{c3 <= 6}}\label{sec: even-low}

This section provides examples of rank-$2$ stable reflexive sheaves with $c_1=0$ and $2\le c_3\le6$ for each possible spectrum in Table \ref{tab: spec c1=0}. First, we present examples of sheaves with $h^0(F(1))=0$. 

\begin{example}[$c_3 = 2z\leq 8$, spectrum $\{-1^z, 0^{4-z}\}$]
Let $C_1$ and $C_2$ be two smooth elliptic quartic curves intersecting on a reduced zero-dimensional scheme $Z$; let $C = C_1 \cup C_2$, and set $z: =h^0(\calo_Z)$. We further assume that $h^0(\sI_C(3)) = 0$; in particular, $z\leq 4$. Then we have {the exact sequence}
\begin{equation}\label{seq: ex0<6}
 0 \longrightarrow \sI_C \longrightarrow \sI_{C_1} \longrightarrow \OO_{C_2}(-Z) \longrightarrow 0
\end{equation}
which we dualize to get {the exact sequence}
\[
0 \longrightarrow \omega_{C_1} = \OO_{C_1} \longrightarrow \omega_{C} \longrightarrow \omega_{C_2}(Z) = \OO_{C_2}(Z) \longrightarrow 0.
\]
Taking global sections, we obtain
\[ 
H^0(\omega_C) = H^0(\OO_{C_1}) \oplus H^0(\OO_{C_2}(Z)), 
\]
thus $h^0(\omega_C)=z+1$. Moreover, we note that there exists $\xi \in H^0(\omega_C)$ vanishing precisely on $Z$, so that the corresponding extension 
\[
0 \longrightarrow \op3 \longrightarrow F_z(2) \longrightarrow \sI_C(4) \longrightarrow 0
\]
defines a stable reflexive sheaf $F_z$ with $c_1 = 0$, $c_2= 4$, and $c_3 = 2z$. Note that $h^0(F_z(1)) = h^0(\sI_C(3)) = 0$. From \eqref{seq: ex0<6} we also get 
\begin{equation}\label{eq: h1ex0<6}
 h^1(F_z(k-2)) = h^1(\sI_C(k)) = \begin{cases}
 0, & k\leq 0 \\ 4- z, & k=1 \\ 6 - z , & k = 2 \\ 4- z, & k=3 \\ 0 , & k \geq 4
\end{cases}, 
\end{equation}
and $h^2(\sI_C(1)) = 0$. Then $h^1(F_z(-2)) = h^2(F(-1)) = 0$ and neither $1$ nor $-2$ can be in the spectrum of $F_z$. Thus, for $z=1, 2, 3$, the spectrum of $F_z$ is, respectively, $\{-1, 0, 0, 0 \}, \{-1, -1, 0, 0 \}$, and $\{-1, -1, -1, 0 \}$. 

Finally, using \eqref{seq: ex0<6} and the vanishing $h^1(F_z(2))=0$, we obtain $h^0(F_z(2)) = 4+z$. 


\begin{table}[H]
 \centering
 \def\arraystretch{1.5}
 \begin{tabular}{|c|c|c|c|c|c|c|c|c|} \hline
 &$-2$ &$-1$ & $0$ & $1$ & 2 \\ \hline \hline
 $h^0(F(p))$ & $0$ & $0$ & $0$ & $0$ & $4+z$ \\ \hline
 $h^1(F(p))$ & $0$ & $4-z$ & $6-z$ & $4-z$ & $0$ \\ \hline
 $h^2(F(p))$ & $z$ & $0$ & $0$ & $0$ & $0$ \\ \hline
 $h^3(F(p))$ & $0$ & $0$ & $0$ & $0$ & $0$ \\ \hline
 \end{tabular}
\caption{Cohomology table for $F_z$ with spectrum $\{-1^z, 0^{4-z}\}$. }
\label{tab: cohomology 021}
\end{table}

From the construction above, we can produce explicit examples of curves, hence of sheaves in these families using Macaulay2 \cite{M2}. One may verify with the ancillary Macaulay2 file \verb|refshsetup.m2| and the code below that such sheaves are generically unobstructed.

\begin{lstlisting}[language=Macaulay2]
load "refshsetup.m2";
Z = pts 4; -- 4 random points
F = flatten entries ( 
 (gens Z)*(matrix basis(2,Z))*random(R^(10 -degree Z),R^4) 
 );
C = intersect(ideal(F_0,F_1),ideal(F_2,F_3));
F = Serre(C,4); chern F
Ext^2(F,F) -- unobstucted
\end{lstlisting}
    
\end{example}

Proof that these sheaves are unobstructed eludes us at the moment. Nonetheless, by semi-continuity, these examples are enough to conclude that they are generically unobstructed. Hence, these families of sheaves lie on {generically reduced} irreducible components of the expected dimension.

\begin{proposition}\label{prop: evenlow}
For $1 \leq z \leq 3$, the moduli space $\calr(0, 4, 2z)$ has a {generically reduced} irreducible component of the expected dimension $29$, whose general member is a sheaf of spectrum $\{-1^z, 0^{4-z}\}$. 
\end{proposition}

We remark that these families of sheaves do not fill their irreducible components. Note that for each $z$, the family $\calc_z$ of curves constructed above is parameterized by a $G(2, 8-z)$-bundle over $z$ copies of $\p3$. Thus, $\dim \calc_z = 32 -z$. The corresponding family of sheaves has dimension $29-z$, by \eqref{eq: dimform-red}.

Next, we provide examples of sheaves with $h^0(F(1))>0$ and the other spectra, starting with $c_3=2$. 

\begin{example}[$c_3 =2$, spectrum $\{-1, -1, 0, 1\}$]
Let $Y$ be a smooth plane conic and $D$ be the double structure on $Y$ given by {the exact sequence}
\[
0 \longrightarrow \sI_D \longrightarrow \sI_Y \longrightarrow \OO_Y(1) \longrightarrow 0.
\]
Then $D$ is a curve of degree $4$ and $p_a(D) = -3$. Dualizing the sequence above and twisting it by $\op3(2)$, we get: 
\[
0 \longrightarrow \omega_Y(2) \simeq \op1(2) \longrightarrow \omega_D(2) \longrightarrow \omega_Y(1) \simeq \op1 \longrightarrow 0.
\]
Hence, $\omega_D(2)$ has global sections vanishing only in dimension $0$. In addition, we have that $h^0(\omega_D(2))=4$ and $h^1(\omega_D(2))=0$. Then let $C = D\cup L$, where $L$ is a line meeting $D$ at one point, hence
\[
0 \longrightarrow \sI_C \longrightarrow \sI_D \longrightarrow \OO_L(-1) \longrightarrow 0.
\]
We have that $C$ is a curve of degree $5$ and genus $-3$. From this sequence, we also get {the exact sequence}
\[
0 \longrightarrow \omega_D(2) \longrightarrow \omega_C(2) \longrightarrow \omega_L(3) \simeq \op1(1) \longrightarrow 0.
\]
Taking global sections, we obtain, since $h^1(\omega_D(2))=0$: 
\[ 
H^0(\omega_C(2)) = H^0(\omega_D(2)) \oplus H^0(\op1(1)). 
\]
It follows that $h^0(\omega_C(2))=6$ and there exists $\xi \in H^0(\omega_C(2))$ vanishing only in dimension $0$. Therefore, $(C, \xi)$ corresponds to a reflexive sheaf $F(1)$ such that $c_1(F) = 0$, $c_2(F) = 4$, and $c_3(F) = 2$. Note that 
\[
h^1(F(-1)) = h^1(\sI_C) = h^1(\sI_D) = h^0(\OO_Y(1)) = 3.
\]
Then, the spectrum of $F$ is $\{-1, -1, 0, 1\}$.

\begin{table}[H]
 \centering
 \def\arraystretch{1.5}
 \begin{tabular}{|c|c|c|c|c|c|c|c|c|} \hline
 &$-3$ &$-2$ &$-1$ & $0$ & $1$ & 2 & $3$ \\ \hline \hline
 $h^0(F(p))$ & $0$ & $0$ & $0$ & $0$ & $1$ & $7$ & $21$ \\ \hline
 $h^1(F(p))$ & $0$ & $1$ & $3$ & $5$ & $4$ & $2$ & $0$ \\ \hline
 $h^2(F(p))$ & $5$ & $2$ & $0$ & $0$ & $0$ & $0$ & $0$ \\ \hline
 $h^3(F(p))$ & $0$ & $0$ & $0$ & $0$ & $0$ & $0$ & $0$ \\ \hline
 \end{tabular}
 \caption{Cohomology table for $F$ with spectrum $\{-1, -1, 0, 1\}$ as described above. }
 \label{tab: cohomology 022}
\end{table}

With the ancillary Macaulay2 file \verb|refshsetup.m2| and the code below, one can check that the construction below gives an unobstructed reflexive sheaf $F$. 

\begin{lstlisting}[language=Macaulay2]
load "refshsetup.m2"
q = x_0*x_3 + x_1^2+ x_2^2 + x_2*x_3;
D = topComponents((ideal(x_0,q))^2 + ideal(x_0*x_3^3-q*x_2^2));
C = intersect(D, ideal(x_1,x_2));
degree C, genus C
F = Serre(C,2); chern F
Ext^2(F,F) --unobstructed
\end{lstlisting}
  
\end{example}


\begin{example}[$c_3 =4$, spectrum $\{-2, -1, 0, 1\}$]
For $c_3 = 4$, we construct the curve $C$ in the following way. Let $D$ be a double conic of genus $-3$ as before. Then define $C = D\cup L$, where $L$ is a general line in the same plane as $D_{\rm red}$. We get the following exact sequence.
\begin{equation}\label{seq: 0441}
 0 \longrightarrow \sI_C \longrightarrow \sI_D \longrightarrow \OO_L(-2) \longrightarrow 0
\end{equation}
In particular, $p_a(C) = \chi(\sI_C) = \chi(\sI_D) - \chi(\OO_L(-2)) = -2$. Dualizing this sequence, we get
\[
0 \longrightarrow \omega_D \longrightarrow \omega_C \longrightarrow \omega_L(2) \simeq \OO_L \longrightarrow 0, 
\]
hence $H^0(\omega_C(2))$ has sections vanishing only in dimension $0$; and $h^0(\omega_C(2)) = h^0(\omega_D(2)) + h^0(\OO_L) = 7$. Therefore, a general $\xi \in H^0(\omega_C(2))$, $(C, \xi)$ corresponds to a reflexive sheaf $F(1)$, such that $c_1(F) = 0$, $c_2(F) = c_3(F) = 4$. Moreover, taking the intersection of $C$ with its supporting plane $H$ yields
\[
0 \longrightarrow \sI_Y(-1) \longrightarrow \sI_C \longrightarrow \sI_{Z/H}(-3) \longrightarrow 0
\]
where $Y=D_{\rm red}$ and $Z$ is $0$-dimensional of length $4$, cf. \cite[Proposition 2.1]{HS}. From this sequence, one can show that $h^1(F(-1)) = h^1(\sI_C) \geq 3$. Hence, the spectrum is $\{-2, -1, 0, 1\}$.

\begin{table}[H]
 \centering
 \def\arraystretch{1.5}
 \begin{tabular}{|c|c|c|c|c|c|c|c|c|} \hline
 &$-3$ &$-2$ &$-1$ & $0$ & $1$ & 2 & $ 3$ & $4$ \\ \hline \hline
 $h^0(F(p))$ & $0$ & $0$ & $0$ & $0$ & $2$ & $9$ & $23$ & $48$ \\ \hline
 $h^1(F(p))$ & $0$ & $1$ & $3$ & $4$ & $4$ & $3$ & $1$ & $0$\\ \hline
 $h^2(F(p))$ & $6$ & $3$ & $1$ & $0$ & $0$ & $0$ & $0$ & $0$\\ \hline
 $h^3(F(p))$ & $0$ & $0$ & $0$ & $0$ & $0$ & $0$ & $0$ & $0$\\ \hline
 \end{tabular}
 \caption{Cohomology table for $F$ with spectrum $\{-2, -1, 0, 1\}$ as described above. }
 \label{tab: cohomology 042}
\end{table}

For every example we computed in Macaulay2, we got $\extd^2(F, F) = 1$. We wonder whether the component containing these sheaves is singular, oversized, or generically non-reduced. One can compute some examples with the code below: 

\begin{lstlisting}[language=Macaulay2]
load "refshsetup.m2";
q = x_0*x_3 + x_1^2+ x_2^2 + x_2*x_3;
Y = ideal(x_0,q);
D = (Y^2 + ideal(x_0*random(3, R)-q*random(2, R))); 
L = ideal(x_0,x_1);
C = intersect(D, L); degree C, genus C
F = Serre(C, 2); chern F
Ext^2(F, F) 
\end{lstlisting}

\end{example}

\begin{example}[$c_3 =6$, spectrum $\{-2, -1, 0, 0\}$]
For $c_3 = 6$, we start with a double structure of genus $-2$ on a conic $Y$ given by 
\begin{equation}\label{seq: 0461}
 0 \longrightarrow \sI_D \longrightarrow \sI_Y \longrightarrow \OO_Y(P) \longrightarrow 0,
\end{equation}
where $P\in Y$ is a point. Then, let $C = D \cup L$, where $L$ is a general line in the same plane as $Y$. As in the previous case, we have
\[
0 \longrightarrow \sI_C \longrightarrow \sI_D \longrightarrow \OO_L(-2) \longrightarrow 0.
\]
In particular $p_a(C) = p_a(D) +1 = -1$. Dualizing this sequence, we get
\[
0 \longrightarrow \omega_D \longrightarrow \omega_C \longrightarrow \omega_L(2) \simeq \OO_L \longrightarrow 0.
\]
Dualizing the sequence \eqref{seq: 0461} we also get
\[
0 \longrightarrow \omega_Y(2) \simeq \OO_{\p1}(2) \longrightarrow \omega_D(2) \longrightarrow \omega_Y(1-P) \simeq \OO_{\p1}(1) \longrightarrow 0.
\]
We then conclude that there exist $\xi\in H^0(\omega_C(2))$ vanishing only in dimension $0$. Moreover, $h^0(\omega_C(2)) = h^0(\omega_D(2)) + 1 = 6$. Therefore, $(C, \xi)$ corresponds to a reflexive sheaf $F(1)$ such that $c_1(F) = 0$, $c_2(F)=4$, and $c_3(F) = 6$. As in the previous case, we intersect $C$ with its supporting plane $H$ to get
\[
0 \longrightarrow \sI_Y(-1) \longrightarrow \sI_C \longrightarrow \sI_{Z/H}(-3) \longrightarrow 0
\]
where $Z$ is $0$-dimensional of length $3$. Therefore, $h^1(F(-1)) = h^1(\sI_C) \geq 2$ and $F$ is of spectrum $\{-2, -1, 0, 0\}$.

\begin{table}[H]
 \centering
 \def\arraystretch{1.5}
 \begin{tabular}{|c|c|c|c|c|c|c|c|} \hline
 &$-2$ &$-1$ & $0$ & $1$ & 2 & $ 3$ \\ \hline \hline
 $h^0(F(p))$ & $0$ & $0$ & $0$ & $2$ & $9$ & $23$ \\ \hline
 $h^1(F(p))$ & $0$ & $2$ & $3$ & $3$ & $2$ & $0$ \\ \hline
 $h^2(F(p))$ & $3$ & $1$ & $0$ & $0$ & $0$ & $0$ \\ \hline
 $h^3(F(p))$ & $0$ & $0$ & $0$ & $0$ & $0$ & $0$ \\ \hline
 \end{tabular}
 \caption{Cohomology table for $F$ with spectrum $\{-2, -1, 0, 0\}$ as described above. }
 \label{tab: cohomology 062}
\end{table}

The code below can produce examples of unobstructed sheaves.


\begin{lstlisting}[language=Macaulay2]
load "refshsetup.m2"
q = x_0*x_3 + x_1^2+ x_2^2 + x_2*x_3;
p1 = x_1*random(1,R)+x_2*random(1,R)+x_3*random(1,R);
p2 = x_1*random(2,R)+x_2*random(2,R)+x_3*random(2,R);
D = topComponents ideal(x_1^2, x_1*q, q^2, x_1*p2-q*p1);
degree D, genus D
C = intersect(D, ideal(x_0,x_1)); degree C, genus C
F = Serre(C,2);
chern F
Ext^2(F,F) --unobstructed
\end{lstlisting}
  
\end{example}

\begin{remark} Every example computed in this section and throughout this work is included in the ancillary file \verb|examples.m2|. Despite our efforts, we were unable to construct sheaves with spectra $\{-1, -1, 0, 1\}$, $\{-2, -1, 0, 1\}$, or $\{-2, -1, 0, 0\}$, such that $h^0(F(1)) = 0$. Neither could we prove they do not exist. 
\end{remark}


\section{Stable reflexive sheaves with \texorpdfstring{$c_1 = 0$}{c1 = 0} and \texorpdfstring{$c_3=8$}{c3=8}}\label{sec: even-8}

Let $F$ be a stable rank-$2$ reflexive sheaf with $c_1 = 0$, $c_2 = 4$, and $c_3 = 8$. By Riemann-Roch and the possible spectra in Table \ref{tab: spec c1=0}, we have that $h^0(F(1))= h^1(F(1))$, $h^2(F) = 0$ and $h^1(F(-1)) \leq 1$. By Lemma \ref{lem: h0F1>0}, $h^0(F(1)) \leq 2$. With this fact in mind, we set up the notation
\[
\calr(0, 4, 8)_{l, m} \coloneqq \left\{ F \in\calr(0, 4, 8) \mid h^0(F(1))=l, \, h^1(F(-1)) = m \right\}, 
\]
where $0\leq l\leq 2$ and $0 \leq m\leq 1$; so that 
\[
\calr(0, 4, 8) = \bigsqcup_{l, m} \calr(0, 4, 8)_{l, m}.
\]
Note that the spectrum is constant in each piece: $\{-1, -1, -1\}$ for $\calr(0, 4, 8)_{l, 0}$, and $\{-2, -1, 0\}$ for $\calr(0, 4, 8)_{l, 1}$.

\begin{theorem}\label{thm: 048}
 The space $\calr(0, 4, 8)$ is generically smooth of expected dimension $29$. Moreover, $\calr(0, 4, 8)_{\rm red} = \overline{\calr(0, 4, 8)_{0, 0}}$, $\calr(0, 4, 8)_{0, 1} = \emptyset$, and the other $\calr(0, 4, 8)_{l, m}$ are hypersurfaces.
\end{theorem}

\begin{table}[H]
 \centering
 \def\arraystretch{1.5}
 \begin{tabular}{|c|c|c|c|c|c|c|c|c|} \hline
 &$-3$ &$-2$ &$-1$ & $0$ & $1$ & 2 & $ 3$ \\ \hline \hline
 $h^0(F(p))$ & $0$ & $0$ & $0$ & $0$ & $l$ & $8+k$ & $24$ \\ \hline
 $h^1(F(p))$ & $0$ & $0$ & $m$ & $2$ & $l$ & $k$ & $0$ \\ \hline
 $h^2(F(p))$ & $8$ & $4$ & $m$ & $0$ & $0$ & $0$ & $0$ \\ \hline
 $h^3(F(p))$ & $0$ & $0$ & $0$ & $0$ & $0$ & $0$ & $0$ \\ \hline
 \end{tabular}
 \caption{Cohomology table for $F\in \calr(0, 4, 8)_{l, m}$, where $k=m(l-1)$. }
 \label{tab: cohomology 048}
\end{table}

\subsection{Description of \texorpdfstring{$\calr(0, 4, 8)_{0, m}$}{R(0, 4, 8) 0, m} } If $h^0(F(1)) = 0$, then $h^1(F(1)) = 0$ and $F$ is $2$-regular due to Castelnuovo-Mumford's criterion. Hence, we get
\begin{equation}\label{seq: 0480}
 0 \longrightarrow E \longrightarrow \op3(-2)^{\oplus 8} \longrightarrow F \longrightarrow 0
\end{equation}
where $E$ is a rank $6$ locally free sheaf. It follows that $E$ is $3$-regular, hence
\[
0 \longrightarrow E' \longrightarrow \op3(-3)^{\oplus 8} \longrightarrow E \longrightarrow 0,
\]
Where $E'$ is a direct sum of line bundles. From a Chern class computation, we get $E' = \op3(-4)^{\oplus 2}$. We conclude, cf. \cite[proof of Lemma 34]{GJM}, that $E = \tp3(-4)^{\oplus 2}$ and, in particular, we have $h^2(F(-1)) = h^3(E(-1)) = 0$, hence the spectrum of $F$ is $\{-1, -1, -1, -1\}$. This implies that $\calr(0, 4, 8)_{0, 1} = \emptyset$. 

On the other hand, applying Lemmas \ref{lem: ext2=0} and \ref{lem: lift} to the family described by \eqref{seq: 0480} we get that $\calr(0, 4, 8)_{0, 0}$ is smooth and unirational of dimension $29$. Hence, the closure $\overline{\calr(0, 4, 8)_{0, 0}}$ is an irreducible component of $\calr(0, 4, 8)$.

\subsection{Description of \texorpdfstring{$\calr(0, 4, 8)_{1, m}$}{R(0, 4, 8) 1, m}} \label{sec: 048-1m}
Let $ F \in \calr(0, 4, 8)_{1, m}$ and consider a curve $C$ defined by a section in $H^0(F(1))$. Note that $C$ has degree $5$ and genus $0$, and satisfies: $h^0(\sI_C(2)) = 0$, $h^0(\sI_C(3)) \geq 4$, and $h^1(\sI_C) = m$. We have two cases: either there exist coprime elements $f, g \in H^0(\sI_C(3))$, or there exists a linear form $h$ dividing every element of $H^0(\sI_C(3))$. The other possibility would be having a degree two polynomial $q$ dividing $H^0(\sI_C(3))$, but this would imply $q \in H^0(\sI_C(2))$ since $h^0(\sI_C(3))\geq 4$.

\subsubsection{ Case 1} Assume that there exist $f, g \in H^0(\sI_C(3))$ relatively prime. Then $V(f, g)$ links $C$ to a curve $Y$ of degree $4$ and genus $-1$, which are described in \cite[Proposition 6.1]{NS-deg4}. The Hilbert Scheme $\hilb^{4,-1}(\p3)$ has three irreducible components: $\calh_1$ composed of extremal curves; $\calh_2$ composed of subextremal curves, the general member is the disjoint union of two conics; $\calh_3$ whose general member is the disjoint union of a twisted cubic and a line. In our case, $h^1(\sI_C(2)) = h^1(F(1)) = 1$, then, by \ref{P2}, we have $h^1(\sI_Y) = 1 < 2$ hence $Y$ cannot be extremal. Therefore, we are left with two cases to deal with.

\paragraph{$\calh_3$ Twisted cubic plus line}\label{case: twc+line} Let $Y$ be a (possibly degenerate) disjoint union of a twisted cubic and a line, then $M_Y \simeq R/(x_0, x_1, x_2^2, x_2x_3, x_3^2)$, cf. the proof of \cite[Proposition 4.2]{NS-deg4}. By \ref{P4}, we have that $M_C = \Ext^4_R(M_Y, R(-6))$, hence we can compute the resolution
\[
 0 \rightarrow R(-6) \overset{B}{\rightarrow} \begin{matrix} R(-5)^{\oplus 2}\\ \oplus\\ R(-4)^{\oplus 3} \end{matrix} \rightarrow \begin{matrix} R(-4) \\ \oplus \\ R(-3)^{\oplus 8} \end{matrix} \rightarrow R(-2)^{\oplus 7} \rightarrow R(-1)^{\oplus 2} \rightarrow M_C \rightarrow 0.
\]
In particular, $m = h^1(\sI_C) = 0$. Following \cite[II \S 5.1]{MDP} we can build a resolution for $\sI_C$; this construction relies on knowing the postulation character $\gamma_C = \{-1, -1, -1, 3, 1, -1\}$, which can be computed using \ref{P3}. We have then
\begin{equation}\label{res: 08twc+line}
 0 \longrightarrow \op3(-6) \overset{B\oplus 0}{\longrightarrow} \begin{matrix}
 \op3(-5)^{\oplus 2} \\ \oplus \\ \op3(-4)^{\oplus 3+ c}
 \end{matrix} \overset{A}{\longrightarrow} \begin{matrix}
 \op3(-3)^4 \\ \oplus \\ \op3(-4)^{\oplus c+1}
 \end{matrix} \longrightarrow \sI_C \longrightarrow 0
\end{equation}
for some $c \geq 0$. We may write the matrix $A$ as a block matrix, 
\[
A = \begin{bmatrix}
 A_{11}^{4\times 2} & A_{12}^{4\times c+3 } \\ A_{21}^{c+1 \times 2} & 0^{c+1\times c+3}
\end{bmatrix},
\]
Where the formats of the blocks are indicated in superscript. Note that $A$ must have generic rank $4+c$. But, on the other hand, $A_{21}$ has generic rank at most $1$, hence $c \leq 1$. If $c=1$ then $A$ has rank at most $4$ along the quartic surface $V(\det A_{12})$, hence $c=0$.

The scheme $\calh_3$ parameterizing disjoint unions of a twisted cubic and a line is irreducible of dimension $12+4 = 16$. For any $Y\in \calh_3$, we have that $h^0(\sI_Y(3)) = 6$. Thus, the family $\mathcal{G}_1$ of curves $C$ described by \eqref{res: 08twc+line} is irreducible of dimension 
\[
\dim \mathcal{G}_1 = 16 + 2h^0(\sI_Y(3)) - 2h^0(\sI_C(3)) = 20, 
\]
cf. \eqref{eq: dimform-liaison}. We remark that a general member of $\mathcal{G}_1$ is a smooth rational quintic curve not contained in a quadric. Indeed, the moduli space of smooth quintics also has dimension $20$, and a general $X$ satisfies $h^0(\sI_X(2))= 0$. By the Riemann-Roch Theorem, we have $h^0(\sI_X(3))\geq 4$, and there are two cubics without a common factor containing $X$. Choosing two of these cubic surfaces, we link $X$ to a curve $Y$ of degree $4$ and genus $-1$ satisfying $h^1(\sI_Y(2)) = h^1(\sI_X) = 0$ hence $Y \in \calh_3$. Therefore, for $C\in \mathcal{G}_1$ we get 
\[
h^0(\omega_C(2)) = h^1(\OO_C(-2)) = h^1(\OO_{\p1}(-10)) = 9.
\]

\paragraph{$\calh_2$ Two conics}\label{case: twoconics} Let $Y$ be a (possibly degenerate) disjoint union of two plane conics, then $Y$ is subextremal with Rao Module $M_Y \simeq R/(x_0, x_1, p,q)$, where $p$ and $q$ are quadratic polynomials in the variables $x_2,x_3$. It follows that $M_C = \Ext^4_R(M_Y, R(-6)) \cong M_Y$ and its minimal resolution is given by the Koszul complex: 
\[
 0 \rightarrow R(-6) \overset{B}{\rightarrow} \begin{matrix} R(-5)^{\oplus 2}\\ \oplus\\ R(-4)^{\oplus 2} \end{matrix} \rightarrow \begin{matrix} R(-4) \\ \oplus \\ R(-3)^{\oplus 4} \\ \oplus \\ R(-2)\end{matrix} \rightarrow \begin{matrix} R(-2)^{\oplus 2}\\ \oplus\\ R(-1)^{\oplus 2} \end{matrix} \rightarrow R \rightarrow M_C \rightarrow 0.
\]
In particular, $m= h^1(\sI_C) = 1$.

As before, $\gamma_C = \{-1, -1, -1, 3, 1, -1\}$ and we use \cite[II \S 5.1]{MDP} to construct a resolution for $\sI_C$: 
\begin{equation}\label{res: 08twoconics}
 0 \longrightarrow \op3(-6) \overset{B\oplus 0}{\longrightarrow} \begin{matrix}
 \op3(-5)^{\oplus 2} \\ \oplus \\ \op3(-4)^{\oplus 2+ c}
 \end{matrix} \overset{A}{\longrightarrow} \begin{matrix}
 \op3(-3)^4 \\ \oplus \\ \op3(-4)^{\oplus c}
 \end{matrix} \longrightarrow \sI_C \longrightarrow 0
\end{equation}
where we may argue as before, analyzing $A$, to conclude that $c\leq 1$. Any curve with a resolution as above with $c=1$ is a specialization of one with $c=0$, cf. \cite[Theorem 4.1]{Kleppe-curves}. Let us denote by $\mathcal{G}_2$ the family of curves described by \eqref{res: 08twoconics}.

For $c=0$, a general curve $C$ in $\mathcal{G}_2$ is the disjoint union of an elliptic quartic and a line. Indeed, let $h_1, h_2 \in H^0(\op3(1))$ and $q_1, q_2 \in H^0(\op3(2))$ such that $Y = V(h_1, q_1) \cup V(h_2, q_2)$. Then, $C$ is residual to $Y$ inside $V(h_1q_2, h_2q_1)$. It follows that $C = V(h_1, h_2) \cup V(q_1, q_2)$. In particular, 
\[
\dim \mathcal{G}_2 = 4+16 = 20, 
\]
and for a general curve $C = L \cup C'$, $C'$ elliptic quartic, we have
\[
h^0(\omega_C(2)) = h^0(\omega_L(2)) + h^0(\omega_{C'}(2)) = 9.
\]

For $c=1$ we may write $B\oplus 0 = \begin{bmatrix} x_0 & x_1 & q_1 & q_2 & 0\end{bmatrix}^T$, where $\deg q_1 = \deg q_2 = 2$. Since $A\cdot (B\oplus 0) = 0$, we can write $A$ as
\[
A = \begin{bmatrix}
 q_2 & 0 & 0 & -x_0 & l_1 \\
 q_1 & 0 & -x_0 & 0 & l_2 \\
 0 & q_2 & 0 & -x_1 & l_3 \\
 0 & q_1 & -x_1 & 0 & l_4 \\
 x_1 & -x_0 & 0& 0& 0
\end{bmatrix}
\]
where $l_{j}$ are linear forms. Then the $4\times 4$ minors of $A$ span the ideal $I$ whose saturation is 
\[
I^{\rm sat} = ( q_{24}x_0, q_{24}x_1 , q_{13}x_0, q_{13}x_1, q_{13}q_1 -q_{24}q_2)
\]
where $q_{13} = l_3x_0-l_1x_1$ and $q_{24} = l_4x_0-l_2x_1$. When the $l_j$ are general enough, $I$ defines the union of the twisted cubic $C' = V(q_{13}, q_{24}, l_1l_4-l_2l_3)$ and the double line $C'' = V(x_0^2, x_0x_1, x_1^2, q_{13}q_1 -q_{24}q_2)$ of genus $-3$. Note that $C'\cap C''$ must have length $4$ so that $p_a(C) = 0$. Then $C'$ is tangent to $C''$, although transverse to $C''_{\rm red}$.

\subsubsection{ Case 2}\label{case: hdivides} Now assume that $H^0(\sI_C(3))$ is spanned by $hq_1, \dots , hq_r$ with $r\geq 4$ and $q_j \in H^0((\op3(2))$ relatively prime. Let $H = V(h)$ and $Y$ be the residual scheme to $C\cap H$ in $C$. Then we have
\begin{equation}\label{seq: hdivides}
 0 \longrightarrow \sI_Y(-1) \longrightarrow \sI_C \longrightarrow \sI_{Z/H}(d-5) \longrightarrow 0
\end{equation}
where $d = \deg Y$ and $Z\subset Y\cap H$. Note that $Y\subset V(q_1, \dots, q_r)$ and is not a plane curve, since $h^0(\sI_C(2)) = 0$. Since $q_1, \dots, q_r$ are relatively prime and $r\geq 4$, we have $d \leq2$. Hence $d=2$ and $Y$ has genus $-1$; indeed, $h^0(\sI_Y(1)) = 0$ implies $p_a(Y)<0$, but for $p_a(Y) \leq -2$ it would have $h^0(\sI_Y(2)) = 3$, cf. \cite[\S 1]{N-deg3}. It follows from the sequence above that $0 = p_a(C) = 1-h^0(\OO_Z) $, hence, $Z$ is a (reduced) point. Moreover, from $h^1(\sI_Y(-1)) = h^2(\sI_Y(-1))=0$ we get 
$m = h^1(\sI_C) = h^1(\sI_{Z/H}(-3)) = 1$. 

A general member of the family of curves described by \eqref{seq: hdivides} is a disjoint union of a line and degenerate elliptic quartic, which is a union of a plane cubic and a line meeting at one point. These are specializations of curves in $\mathcal{G}_2$.

\subsubsection{Summary of the families}
A sheaf $ F \in \calr(0, 4, 8)_{1, 0}$ must correspond to a curve $C\in \mathcal{G}_1$, described in \ref{case: twc+line}. We have that $\mathcal{G}_1$ is irreducible of dimension $20$, and a general member satisfies $h^0(\omega_C(2)) = 9$. Then $\calr(0, 4, 8)_{1, 0}$ is irreducible of dimension
\[
\dim \calr(0, 4, 8)_{1, 0} = \dim \mathcal{G}_1 + h^0(\omega_C(2)) -h^0(F(1)) = 28, 
\]
due to \eqref{eq: dimform-red}. Since a general curve in $\mathcal{G}_1$ is a smooth rational quintic, a general element of $\calr(0, 4, 8)_{1, 0}$ is unobstructed due to \cite[Proposition 1.4]{GM-gensm}. Then we conclude that $\calr(0, 4, 8)_{1, 0}$ is reduced. 

\bigskip

The sheaves $ F \in \calr(0, 4, 8)_{1, 1}$ correspond to curves in the family $\mathcal{G}_2$, described in \ref{case: twoconics}, or in its boundary, described in \ref{case: hdivides}. A general member of $\mathcal{G}_2$ is a disjoint union of a line $L$ and an elliptic quartic $D$. These curves are parameterized by an open subset of $G(2, H^0(\op3(1)))\times G(2, H^0(\op3(2)))$, which has dimension $20$. Moreover, 
\[
h^0(\omega_C(2)) = h^0(\omega_D(2))+h^0(\omega_L(2)) = 9, 
\]
therefore, we conclude that $\calr(0, 4, 8)_{1, 1}$ is irreducible of dimension $28$.

\subsection{Description of \texorpdfstring{$\calr(0, 4, 8)_{2, 0}$}{{R(0, 4, 8) 2, 0}}} Consider a curve $C$ defined by a section in $H^0(F(1))$; it has degree $5$ and genus $0$, and $h^0(\sI_C(2)) = 1$. We also have that $h^1(\sI_C(2)) = h^1(F(1)) = 2$ and $h^1(\sI_C)= h^1(F(-1)) = 0$. From Lemma \ref{lem: curv5quad}, we see that $C$ is a divisor of bidegree $(1, 4)$ on a smooth quadric. Then, 
\[
0 \longrightarrow \op3(-2) \overset{\cdot Q}{\longrightarrow} \sI_C \longrightarrow \OO_Q(-1, -4) \longrightarrow 0, 
\]
and we may form the following commutative diagram {with exact rows and columns}: 
\[
\xymatrix@-2ex{ 
& 0 \ar[d] & 0 \ar[d] & \\
& \op3(-1)\ar[d]\ar@{=}[r] & \op3(-1)\ar[d] & & \\
0\ar[r] & K \ar[r]\ar[d] & F \ar[r] \ar[d] & \OO_Q(0, -3) \ar@{=}[d] \ar[r] & 0 \\
0\ar[r] & \op3(-1) \ar[r]\ar[d] & \sI_C(1) \ar[r]\ar[d] & \OO_Q(0, -3) \ar[r] & 0 \\
& 0 & 0 & 
}
\]
It follows that $K = \op3(-1)^{\oplus 2}$ and $F$ is given by an extension of $\OO_Q(0, -3)$ by $\op3(-1)^{\oplus 2}$. A general map $\OO_Q(0, -3) \to \OO_Q$ defines a curve $Y$ that is the disjoint union of $3$ lines. Then, applying $\Hom( \, - , \op3(-1)^{\oplus 2})$ to the {exact} sequence
\[
0 \longrightarrow \OO_Q(0, -3) \longrightarrow \OO_Q \longrightarrow \OO_Y \longrightarrow 0
\]
we get
\begin{multline*}
 0 \to \Ext^1(\OO_Q, \op3(-1)^{\oplus 2}) \to \Ext^1(\OO_Q(0, -3), \op3(-1)^{\oplus 2}) \to \\ \to \Ext^2(\OO_Y , \op3(-1)^{\oplus 2}) \to \Ext^2(\OO_Q, \op3(-1)^{\oplus 2}) = 0
\end{multline*}
Note that $\Ext^1(\OO_Q, \op3(-1)^{\oplus 2}) \simeq H^0(\op3(1))^{\oplus 2}$ and $\Ext^2(\OO_Y , \op3(-1)^{\oplus 2}) \simeq H^0(\omega_Y(3))^{\oplus 2}$, hence $\extd^1(\OO_Q(0, -3), \op3(-1)^{\oplus 2}) = 20$. Since the quadric $Q$ is parameterized by an open subset of $\p9$, we have that
\[
\dim \calr(0, 4, 8)_{2, 0} = 9+19 = 28.
\]

\subsection{Description of \texorpdfstring{$\calr(0, 4, 8)_{2, 1}$}{R(0, 4, 8) 2, 1}} As in the previous case, $F(1)$ corresponds to a curve of genus $0$ and degree $5$ contained in a quadric. But now, $h^1( \sI_C) = h^1(F(-1)) = 1$, and Lemma \ref{lem: curv5quad} implies that $C$ is subextremal. We have that $C$ is contained in the (possibly degenerate) union of two planes $H_1\cup H_2$
\[
0 \longrightarrow \sI_Y(-1) \longrightarrow \sI_C \longrightarrow \sI_{Z/H_2}(-3) \longrightarrow 0 
\]
where $Y\subset H_1$ is a plane conic, and $Z$ has length $2$. As in the previous case, we build the {exact commutative} diagram 
\[
\xymatrix@-2ex{ 
& 0 \ar[d] & 0 \ar[d] & \\
& \op3(-1)\ar[d]\ar@{=}[r] & \op3(-1)\ar[d] & & \\
0\ar[r] & K \ar[r]\ar[d] & F \ar[r] \ar[d] & \sI_{Z/H_2}(-2) \ar@{=}[d] \ar[r] & 0 \\
0\ar[r] & \sI_Y \ar[r]\ar[d] & \sI_C(1) \ar[r]\ar[d] & \sI_{Z/H_2}(-2) \ar[r] & 0 \\
& 0 & 0 & 
}
\]
It follows that $K\in \calr(-1, 2, 4)$ and is given by a resolution of the form 
\begin{equation}\label{res: auxK}
 0 \longrightarrow \op3(-3) \longrightarrow \begin{matrix}
 \op3(-1)^{\oplus 2} \\ \oplus \\ \op3(-2)
\end{matrix} \longrightarrow K \longrightarrow 0 
\end{equation}
Applying $\Hom( \OO_Z, \, - )$ we get the injectivity of
\[
0 \longrightarrow \Ext^3(\OO_Z , \op3(-3)) \longrightarrow \Ext^3(\OO_Z , \op3(-1)^{\oplus 2} \oplus \op3(-2)) 
\]
Therefore, noting that $\Ext^2( \sI_{Z/H_2}, \op3(l) ) \cong \Ext^3(\OO_Z , \op3(l)) $ for $l \geq -3$, we apply $\Hom( \sI_{Z/H_2}(-2), \, - )$ to \eqref{res: auxK} to get {exact sequence}
\begin{multline*}
 0\to H^0( \OO_{H_2})\to H^0(\OO_{H_2}(2)^{\oplus 2} \oplus \OO_{H_2}(1)) \to \Ext^1( \sI_{Z/H_2}(-2), K) \to 0
\end{multline*}
hence $\extd^1(\sI_{Z/H_2}(-2), K) = 14$. On the other hand, $\calr(-1, 2, 4)$ has dimension $11$, cf. \cite[Theorem 9.2]{H2}. Given $K$, we have a $1$-dimensional family for $Y$. Moreover, choosing $H_2$ determines $Z = H_{2}\cap Y$. This amounts to a total of $15$ parameters. We conclude that
\[
\dim \calr(0, 4, 8)_{2, 1} = 13+15 = 28.
\]


\section{Stable reflexive sheaves with \texorpdfstring{$c_1 = 0$}{c1 = 0} and \texorpdfstring{$c_3=10$}{c3=10}}\label{sec: even-10}

Hartshorne observed in \cite[Example 5.1.3]{H4} that the spectrum $\{-2, -2, -1, 0\}$ does not occur. Therefore, the only realizable spectrum is $\{-2, -1, -1, -1\}$. Below is the cohomology table for a stable rank-$2$ reflexive sheaf $F\in \calr(0, 4, 10)$.

\begin{table}[H]
\centering
\def\arraystretch{1.5}
\begin{tabular}{|c|c|c|c|c|c|c|c|c|} \hline
 &$-2$ &$-1$ & $0$ & $1$ & 2 \\ \hline \hline
 $h^0(F(p))$ & $0$ & $0$ & $0$ & $l+1$ & $9$ \\ \hline
 $h^1(F(p))$ & $0$ & $0$ & $1$ & $l$ & $0$ \\ \hline
 $h^2(F(p))$ & $5$ & $1$ & $0$ & $0$ & $0$ \\ \hline
 $h^3(F(p))$ & $0$ & $0$ & $0$ & $0$ & $0$ \\ \hline
\end{tabular}
\caption{Cohomology table for stable rank-$2$ reflexive sheaf $F\in \calr(0, 4, 10)$ and spectrum $\{-2, -1, -1, -1\}$; $l\coloneqq h^1 (F(1))\le1$.}
 \label{tab: cohomology 0101 gen}
\end{table}

Due to Lemma \ref{lem: h0F1>0}, we have that $h^0(F(1)) = h^1(F(1)) + 1 \leq 2$, thus $h^1(F(1)) \leq 1$. With this fact in mind, we set up the notation
\[ 
\calr(0, 4, 10)_{l} \coloneqq \{ F \in\calr(0, 4, 10) \mid h^1(F(1))=l \}, 
\]
so that $\calr(0, 4, 10)=\calr(0, 4, 10)_0 \sqcup \calr(0, 4, 10)_1$. In this section, we prove the following statement by analyzing $\calr(0, 4, 10)_0$ and $\calr(0, 4, 10)_1$ individually.

\begin{theorem}\label{thm: 0410}
$\calr(0, 4, 10)$ is a smooth, irreducible, unirational quasi-projective variety of dimension 29. In addition, $\calr(0, 4, 10)_0$ is an open subset, while $\calr(0, 4, 10)_1$ has codimension 2.
\end{theorem}

\begin{proof}
Notice from Table \ref{tab: cohomology 0101 gen} that every $F\in\calr(0, 4, 10)$ is 3-regular and satisfies $h^0(F)=h^1(F(-1))=0$. Therefore, we can invoke Lemma \ref{lem: ext2=0} with $k=3$ to conclude that $\Ext^2(F, F)=0$ for every $F\in\calr(0, 4, 10)$, thus showing that $\calr(0, 4, 10)$ is a smooth quasi-projective variety of dimension 29.

We will argue below that $\calr(0, 4, 10)_0$ is a 29-dimensional quasi-projective variety that admits a surjective morphism from an affine space, while $\dim\calr(0, 4, 10)_1=27$. {Since the local dimension of $\calr(0, 4, 10)$ at any closed point of is $29$, $\calr(0, 4, 10)_1$ must be contained in the closure of $\calr(0, 4, 10)_0$.}
Thus, $\calr(0, 4, 10)$ must be irreducible, and unirationality follows from the existence of a dominant morphism $\mathbb{A}^{80}\to\calr(0, 4, 10)$.
\end{proof}


\subsection{Description of \texorpdfstring{$\calr(0, 4, 10)_0$}{R(0, 4, 10) 0}}

Let $\sigma$ be a generator of $V_1\coloneqq H^0(F(1))$. Then the 9-dimensional space $H^0(F(2))$ has a basis of the form 
\[
\{x\sigma, y\sigma, z\sigma , w\sigma, \theta_1, \dots, \theta_5 \}.
\]
Setting $V_5\coloneqq \langle \theta_1, \dots, \theta_5 \rangle$, we get an epimorphism $V_1\otimes\op3(-1)\oplus V_5\otimes\op3(-2) \twoheadrightarrow F$ given by $\begin{bmatrix}
\sigma & \theta_1 & \dots & \theta_5
\end{bmatrix}$. Let $E$ be the kernel of this morphism, giving the short exact sequence
\begin{equation}\label{eq: res1}
0 \longrightarrow E \longrightarrow V_1\otimes\op3(-1)\oplus V_5\otimes\op3(-2) \longrightarrow F \longrightarrow 0.
\end{equation}
Note that $E$ is locally free, since $\inext^p(E, \op3)=0$ for $p\ge1$, with $\rk(E)=4$ and $c_1(E)=-11$. In addition, it follows from the short exact sequence above that $h^1(E(2))=0$. Hence, $E$ is $3$-regular. Since $h^1(E(3))=0$, we also obtain that $h^0(E(3))=h^0(\op3(2))+5\cdot h^0(\op3(1))-h^0(F(3))=5$. Therefore, we obtain the short exact sequence 
\[
0 \longrightarrow \op3(-4) \overset{\mu}{\longrightarrow}\op3(-3)^{\oplus 5} \overset{\eta}{\longrightarrow} E \longrightarrow 0
\]
because $\ker\eta$ must be a line bundle with $c_1(\ker\eta)=-15-c_1(E)=-4$. Since $E$ is locally free, the entries of $\mu$ must generate the irrelevant ideal. {Up to acting with $\gl(5,\kappa)$, $\mu$ is given by multiplication with the vector $\begin{bmatrix} x_0 & x_1 &x_2 &  & x_3 & 0 \end{bmatrix}$}. 
Therefore, $E = \tp3(-4) \oplus \op3(-3)$ and every $F\in\calr(0, 4, 10)_0$ admits a locally free resolution of the following form
\begin{equation}\label{eq: res even-10}
0 \longrightarrow \tp3(-4) \oplus \op3(-3) \overset{\varphi}{\longrightarrow} \op3(-1)\oplus \op3(-2)^{\oplus 5} \longrightarrow F \longrightarrow 0.
\end{equation}
{Applying Lemma \ref{lem: lift}, we see that} the image of the induced modular morphism $\Psi\colon \mathbb{M}_0\to\calr(0, 4, 10)$ is isomorphic to $\calr(0, 4, 10)_{0}$. Moreover, notice that: 
\[ 
\dim\mathbb{M}_0 - \dim \mathbb{G} = 80 - 51 = 29 = \extd^1(F_\varphi, F_\varphi) ~~, ~~ \forall \varphi\in\mathbb{M}_0 .
\]

\begin{remark}
Note that the unique (up to scalar multiplication) nontrivial section in $H^0(F(1))$ vanishes along a connected curve of degree $5$ and genus $1$ not contained in a quadric. Generically, this is a smooth elliptic quintic.
\end{remark}


\subsection{Description of \texorpdfstring{$\calr(0, 4, 10)_1$}{R(0, 4, 10) 1}} 

Let $C$ be a curve defined by a nontrivial global section of $F(1)$. Then $C$ has degree $5$ and genus $1$ and is contained in a quadric. Moreover, $C$ is not extremal ($h^0(\sI_C(2))=1$) and $h^1(\sI_C(1))= h^1(\sI_C(2)) = 1$. Hence, $C$ is subextremal and, in particular, $h^1(\sI_C(l)) = 0$ for $l\not\in \{1, 2\}$. Due to \cite[Proposition 9.9]{HS}, there exists a plane $H$ and an exact sequence
\begin{equation}\label{eq: YCZ}
0 \longrightarrow \sI_Y(-1) \overset{h}{\longrightarrow} \sI_C \longrightarrow \sI_{Z/H}(-3) \longrightarrow 0, 
\end{equation}
where $Y$ is a plane conic and $Z\subset H$ is a zero-dimensional subscheme. Moreover, $h^0(\OO_Z) = p_a(C) = 1$ hence $Z$ is a point. Note that $C$ is contained in a pair of surfaces of degrees $2$ and $4$ without common factors. Hence, $C$ is linked to a curve $\Gamma$ of degree $3$ and genus $-1$. The Hilbert Scheme $\hilb^{3, -1}(\p3)$ is smooth and irreducible of dimension $12$, cf. \cite[Théorème 4.1]{MDP2}. Then, the family of curves $C$ satisfying \eqref{eq: YCZ} fill the Hilbert Scheme of constant cohomology $H_{\gamma, \rho}$, where the postulation $\gamma$ and the Rao function $\rho$ are determined by \eqref{eq: YCZ}. Moreover, 
\begin{align*}
 \dim H_{\gamma, \rho} &= \dim\hilb^{3, -1}(\p3) + h^0(\sI_\Gamma (2)) + h^0(\sI_\Gamma (4)) - h^0(\sI_C (2)) - h^0(\sI_C (4))  \\
 & = 12 + 2 + 21 - 1 -15 = 19.
\end{align*}

From \eqref{eq: YCZ} we can also compute $h^0(\omega_C(2))$. Note that $\inext^1(\sI_{Z/H}(-3), \opn) = \OO_H(4)$, and $\inext^2(\sI_{Z/H}(-3), \opn) = \OO_Z$. Dualizing and twisting the sequence \eqref{eq: YCZ}, we get 
\[
0 \longrightarrow \OO_W(2) \longrightarrow \omega_C(2) \longrightarrow \omega_Y(3) \longrightarrow \OO_Z \longrightarrow 0,
\]
where $W$ is a plane cubic. Hence, 
\[
h^0(\omega_C(2)) = h^0(\OO_W(2)) + h^0(\OO_Y(2H' - Z)) = 6 + 4 = 10,
\]
which yields
\[
\dim \calr(0, 4, 10)_1 = 19 + 10 - h^0(F(1)) = 27.
\]


\section{Stable reflexive sheaves with \texorpdfstring{$c_1 = 0$}{c1 = 0} and \texorpdfstring{$c_3=12$}{c3=12}}\label{sec: even-12}

The possible spectra in this case imply that $h^2(F(l))=0$ for $l\ge1$, so the expression is display \eqref{eq: h0F1} becomes $\chi(F(1)) = h^0(F(1)) - h^1(F(1)) = 2$; thus, $h^0(F(1))\ge2$. On the other hand, Lemma \ref{lem: h0F1>0} implies that $h^0(F(1))=2$ when the spectrum of $F$ is $\{-2, -2, -1, -1\}$, and $h^0(F(1))=3$ when the spectrum of $F$ is $\{-3, -2, -1, 0\}$. Defining the sets
\[ 
\calr(0, 4, 12)_{l} \coloneqq \{ F \in\calr(0, 4, 12) \mid h^1(F(1))=l \}, 
\]
we have that every sheaf in $\calr(0, 4, 12)_{0}$ has spectrum equal to $\{-2, -2, -1, -1\}$, while every sheaf in $\calr(0, 4, 12)_{1}$ has spectrum equal to $\{-3, -2, -1, 0\}$. In addition, the formula in display \eqref{eq: h0F1} becomes
\[ 
\chi(F(2)) = h^0(F(2)) - h^1(F(2)) = 10, 
\]
thus $h^0(F(2))\geq 10$. One can check that every $F\in\calr(0, 4, 12)_{0}$ is 2-regular, thus, $h^1(F(2))=0$. When $F\in\calr(0, 4, 12)_{1}$, we will see that $h^1(F(2))=1$ and $h^1(F(p))=0$ for $p\ge3$. We end up with the following cohomology table.

\begin{table}[H]
 \centering
 \def\arraystretch{1.5}
 \begin{tabular}{|c|c|c|c|c|c|c|c|c|} \hline
 &$-3$ &$-2$ &$-1$ & $0$ & $1$ & 2 & $ 3$ \\ \hline \hline
 $h^0(F(p))$ & $0$ & $0$ & $0$ & $0$ & $2+l$ & $10+l$ & $26$ \\ \hline
 $h^1(F(p))$  & $0$ & $0$ & $l$ & $l$ & $l$ & $l$ & $0$ \\ \hline
 $h^2(F(p))$ & $10$ & $6$ & $2+l$ & $l$ & $0$ & $0$ & $0$ \\ \hline
 $h^3(F(p))$  & $0$ & $0$ & $0$ & $0$ & $0$ & $0$ & $0$ \\ \hline
 \end{tabular}
 \caption{Cohomology table for stable rank-$2$ reflexive sheaves $F$ in $\calr(0, 4, 12)_{l}$, with $l\in \{0,1\}$.}
 \label{tab: cohomology 0121}
\end{table}


This section establishes the following description of $\calr(0, 4, 12)$.

\begin{theorem}\label{thm: 0412}
The moduli scheme $\calr(0, 4, 12)$ has two connected components: 
\begin{itemize}
\item $\calr(0, 4, 12)_0$ is an irreducible, smooth, unirational quasi-projective variety of dimension 29; it consists of all sheaves with spectrum $\{-2, -2, -1, -1\}$;
\item {$\calr(0, 4, 12)_1$ is a non-reduced scheme of dimension 29 whose Zariski tangent spaces have dimension equal to 30}; $\left(\calr(0, 4, 12)_1\right)_{\rm red}$ is irreducible, smooth, and rational; it consists of all sheaves with spectrum $\{-3, -2, -1, 0\}$.
\end{itemize}
\end{theorem}


\subsection{Description of \texorpdfstring{$\calr(0, 4, 12)_0$}{R(0, 4, 12) 0}} 

Let $F\in\calr(0, 4, 12)_0$ and let $\sigma_1, \sigma_2$ be a basis for $H^0(F(1))$. Note that $\{ x_j\sigma_i \}_{0\leq j \leq 3; i = 1, 2} $ form a linearly independent subset of $H^0(F(2))$. Indeed, suppose that there exists a nontrivial relation $\sum a_{ij}x_j\sigma_i =0$; it can be rewritten in the form $l_1\sigma_1+ l_2\sigma_2 =0 $ for some $l_1, l_2 \in H^0(\op{3}(1))$. Let $L$ be the line given by $V(l_1, l_2)$ and consider the exact sequence
\[
\intor_1^{\op3}(F, \sI_L(2)) = 0 \longrightarrow F \longrightarrow F(1)^{\oplus 2} \longrightarrow F(2)\otimes\sI_L \longrightarrow 0.
\]
For the vanishing of $\intor_1^{\op3}(F, \sI_L(2))$ we refer to \cite[Lemma A.1 (ii)]{CJM-deg1}. Taking global sections, we see $l_1\sigma_1+ l_2\sigma_2 = 0$ if and only if there exists $\delta \in H^0(F)$ such that $(\sigma_1, \sigma_2) = (l_2\delta, -l_1\delta)$; but this is impossible since $H^0(F) = 0$. Therefore, using that $F$ is 2-regular, we have a surjective morphism $\op3(-2)^{\oplus 2} \oplus \op3(-1)^{\oplus 2} \twoheadrightarrow F$; 
{let $E$ be its kernel. Computing Chern classes one gets that $c_1(E(3)) = c_2(E(3)) = 0$, hence $E = \op3(-3)^{\oplus 2}$.}
Hence,
\begin{equation} \label{c3=12 res}
0 \longrightarrow \op3(-3)^{\oplus 2} \longrightarrow \op3(-2)^{\oplus 2} \oplus \op3(-1)^{\oplus 2}\longrightarrow F \longrightarrow 0. 
\end{equation}

Lemma \ref{lem: ext2=0} implies that $\extd^2(F, F)=0$, thus $\extd^1(F, F)=29$ and $\calr(0, 4, 10)_0$ is smooth. In addition, a straightforward computation tells us that the family of stable reflexive sheaves with a resolution as in display \eqref{c3=12 res} has dimension 29. Applying Lemma \ref{lem: lift}, we can conclude that the closure of $\calr(0, 4, 12)_0$ within $\calr(0, 4, 12)$ is an irreducible component of the latter. It is unirational because $\calr(0, 4, 12)_0$ admits a dominant morphism from an open subset of the affine space $\Hom(\op3(-3)^{\oplus 2}, \op3(-2)^{\oplus 2} \oplus \op3(-1)^{\oplus 2})\simeq\mathbb{A}^{54}$. 


\subsection{Description of \texorpdfstring{$\calr(0, 4, 12)_1$}{R(0, 4, 12) 1}}

If $h^0(F(1))= 3$ then $F(1)$ must correspond to an extremal curve $C$. Since $\deg(C)=5$ we have that $C$ must contain a planar quartic subcurve and a residual line $L$: 
\[
0 \longrightarrow \sI_L(-1) \longrightarrow \sI_C \longrightarrow \sI_{Z/H}(-4) \longrightarrow 0.
\]
Since $p_a(C) = 2$, we get that $Z$ has length one. This leads us to the following {exact} commutative diagram: 
\begin{equation}\label{diag1}
\begin{split} \xymatrix@-2ex{ 
& 0 \ar[d] & 0 \ar[d] & \\
& \op3(-1)\ar[d]\ar@{=}[r] & \op3(-1)\ar[d] & & \\
0\ar[r] & E \ar[r]\ar[d] & F \ar[r] \ar[d] & \sI_{Z/H}(-3) \ar@{=}[d] \ar[r] & 0 \\
0\ar[r] & \sI_L \ar[r]\ar[d] & \sI_C(1) \ar[r]\ar[d] & \sI_{Z/H}(-3) \ar[r] & 0 \\
& 0 & 0 & 
} \end{split}
\end{equation}
where $E$ is the kernel of the composed morphism $F\to\sI_C(1)\to\sI_{Z/H}(-3)$; the leftmost column tells us that $E$ is a stable reflexive sheaf with Chern classes $(c_1(E), c_2(E), c_3(E))=(-1, 1, 1)$. It follows that every $F$ in $\calr(0, 4, 12)_1$ fits into a short exact sequence of the form
\begin{equation}\label{eq: middle}
0 \longrightarrow E \longrightarrow F \longrightarrow \sI_{Z/H}(-3) \longrightarrow 0
\end{equation}
given by the middle row in diagram \eqref{diag1}.

The description of $F\in\calr(0, 4, 12)_1$ in display \eqref{eq: middle} matches the exact sequence in \cite[((3)), p.103]{Ch3} when $c_2(F)=4$. The main result regarding this family of sheaves is \cite[Theorem 10]{Ch3} that, in our notation, states that the reduced scheme $\left(\calr(0, 4, 12)_1\right)_{\rm red}$ is irreducible, nonsingular, and rational of dimension 29. Due to Proposition \ref{prop: extremal-ext}, $\extd^2(F, F) = 1$, hence $\extd^1(F, F)=30$, for every $F\in\calr(0, 4, 12)_1$.


\section{Examples of stable reflexive sheaves with \texorpdfstring{$c_1 = -1$}{c1 = -1} and \texorpdfstring{$c_3\leq 6$}{c3<=6}}\label{sec: odd-low}

This section provides examples of rank-$2$ stable reflexive sheaves with $c_1=-1$ and $2\le c_2\le6$ for each possible spectrum in Table \ref{tab: spec c1=-1}.

\begin{example}[$c_3 = 2$, spectrum $\{-1, -1, -1, 0\}$] Let $C$ be the disjoint union of two twisted cubics. Then $C$ has degree $6$ and genus $-1$, and a general element of $H^0(\omega_C(1))$ vanishes at one point in each component. The corresponding extension 
\[
0 \longrightarrow \op3 \longrightarrow F(2) \longrightarrow \sI_C(3) \longrightarrow 0
\]
defines a reflexive sheaf $F\in \calr(-1, 4, 2)$ such that $h^1(F(l)) = h^1(\sI_C(l+1)) = 1$ for $l=-1, 0, 1$, and $h^1(F(l)) = 0$ for $|l| \geq 2$. Therefore, the spectrum of $F$ is $\{-1, -1, -1, 0\}$. Note that a disjoint union of a rational quartic curve and a plane conic also defines a sheaf with this same spectrum. We have $h^0(F(1)) = 0$ in both cases. The difference in cohomology is that $h^1(\sI_C(3)) =0$ in the case of twisted cubics, but $h^1(\sI_C(3)) = 1$ in the second case. We can check with the code below that these examples belong to generically smooth irreducible components.

\medskip
\begin{lstlisting}[language=Macaulay2]
load "refshsetup.m2";
C1 = minors(2, random(R^3, R^{-1, -1}));
C2 = minors(2, random(R^3, R^{-1, -1}));
C = intersect(C1, C2); -- disjoint union of twisted cubics
degree C, genus C
F = Serre(C, 3); chern F
Ext^2(F, F) -- unobstructed 
---
C1' = kernel map(kk[s, t], R, {s^4, s^3*t, s*t^3, t^4});
C2' = ideal random(R^1, R^{-1, -2});
C' = intersect(C1', C2'); -- plane conic + rat'l quartic
degree C', genus C'
F' = Serre(C', 3);
Ext^2(F', F') -- unobstructed
\end{lstlisting}
\medskip

\begin{table}[H]
 \centering
 \def\arraystretch{1.5}
 \begin{tabular}{|c|c|c|c|c|c|c|c|c|} \hline
 & $-2$ &$-1$ & $0$ & $1$ & 2 & $3$ \\ \hline \hline
 $h^0(F(p))$ & $0$ & $0$ & $0$ & $0$ & $l+1$ & $13$ \\ \hline
 $h^1(F(p))$ & $0$ & $1$ & $4$ & $4$ & $l$ & $0$ \\ \hline
 $h^2(F(p))$ & $3$ & $0$ & $0$ & $0$ & $0$ & $0$ \\ \hline
 $h^3(F(p))$ & $0$ & $0$ & $0$ & $0$ & $0$ & $0$ \\ \hline
 \end{tabular}
 \caption{Cohomology table for $F\in \calr(-1,4,2)$ with spectrum $\{-1, -1, -1, 0\}$ as in the example. $l\leq 1$. }
 \label{tab: cohomology 121}
\end{table}

\end{example}

\begin{example}[$c_3 = 2$, spectrum $\{-2, -1, 0, 0\}$]
Consider $Y$ a plane conic, and let $p, q, r \in Y$ be $3$ general points. Then let $C$ be the double conic given by
\[
0\longrightarrow \sI_C \longrightarrow \sI_Y \longrightarrow \OO_Y(p+q+r) \longrightarrow 0
\]
so that $p_a(C) = p_a(Y)-\chi(\OO_Y(p+q+r)) = - 4$. Dualizing this sequence, we get
\[
0\longrightarrow \omega_Y \longrightarrow \omega_C \longrightarrow \inext^2(\OO_Y(p+q+r), \op3(-4)) \simeq \omega_Y(-p-q-r)\longrightarrow 0
\]
Then $h^0(\omega_C(3)) = 7$ and, for a general $\xi \in H^0(\omega_C(3))$, the restriction $\xi|_Y$ vanishes at only one point. Therefore, $(C, \xi)$ corresponds to a reflexive sheaf $F(1)$ such that $(c_1(F) , c_2(F), c_3(F)) = (-1, 4, 2)$. Note that $h^1(F(-1)) = h^1(\sI_C(-1)) = 2$, thus the spectrum of $F$ must be $\{-2, -1, 0, 0\}$. With the code below, one can produce an unobstructed example. Hence, these sheaves lie on a generically smooth irreducible component of $\calr(-1,4,2)$ of expected dimension $27$.

\begin{lstlisting}[language=Macaulay2]
load "refshsetup.m2";
q = x_0^2-x_1*x_2;
a = x_3^3*x_0+x_0^3*x_1+x_2^3*x_3;
b = (x_2^2+x_1^2)*x_0+x_1^2*x_1+x_0^2*x_3;
C = topComponents ideal( x_3^2, x_3*q, q^2, x_3*a-q*b);
degree C, genus C
F = Serre(C, 1); chern F
Ext^2(F, F) -- unobstructed
\end{lstlisting}

\begin{table}[H]
 \centering
 \def\arraystretch{1.5}
 \begin{tabular}{|c|c|c|c|c|c|c|c|c|} \hline
 &$-3$ &$-2$ &$-1$ & $0$ & $1$ & 2 & $3$ & 4\\ \hline \hline
 $h^0(F(p))$ & $0$ & $0$ & $0$ & $0$ & $1$ & $5$ & $15$ & $34$ \\ \hline
 $h^1(F(p))$ & $0$ & $0$ & $2$ & $4$ & $5$ & $4$ & $2$ & $0$\\ \hline
 $h^2(F(p))$ & $6$ & $3$ & $1$ & $0$ & $0$ & $0$ & $0$ & $0$\\ \hline
 $h^3(F(p))$ & $0$ & $0$ & $0$ & $0$ & $0$ & $0$ & $0$ & $0$\\ \hline
 \end{tabular}
 \caption{Cohomology table for $F\in \calr(-1,4,2)$ with spectrum $\{-2, -1, 0, 0\}$ as in the example. }
 \label{tab: cohomology 122}
\end{table}

\end{example}

\begin{example}[$c_3 = 4$, spectrum $\{-1, -1, -1, -1\}$]
Let $C$ be a smooth rational sextic curve; this case was suggested in \cite[Example 4.2.4]{H2}. Since $\omega_C(1) \cong \OO_{\p1}(4)$, $C$ corresponds to a sheaf $F(2)$ such that $(c_1(F) , c_2(F), c_3(F)) = (-1, 4, 4)$. Moreover, $h^1(F(-1)) = h^1(\sI_C) = 0$ since $C$ is integral. Then $F$ has spectrum $\{-1, -1, -1, -1\}$. Also note that for a general curve $C$, $h^0(\sI_C(2)) = 0$, hence $h^0(F(1)) = 0$. With the code below, we can show that, generically, $\Ext^2(F, F)= 0$, and these sheaves belong to a generically smooth irreducible component of $\calr(-1, 4, 4)$.

\begin{lstlisting}[language=Macaulay2]
load "refshsetup.m2";
S = kk[y_0..y_6];
Y = kernel map(kk[s, t], S, {t^6, t^5*s, t^4*s^2, t^3*s^3, 
 t^2*s^4, t*s^5, s^6}); 
phi = map(S, R, random(S^{1}, S^4));
C = preimage(phi, Y); degree C, genus C
F = Serre(C, 3); chern F
Ext^2(F, F) --unobstructed
\end{lstlisting}

We can also describe the cohomology table of such sheaves: 

\begin{table}[H]
 \centering
 \def\arraystretch{1.5}
 \begin{tabular}{|c|c|c|c|c|c|c|c|c|} \hline
 &$-3$ &$-2$ &$-1$ & $0$ & $1$ & 2 & $3$ \\ \hline \hline
 $h^0(F(p))$ & $0$ & $0$ & $0$ & $0$ & $0$ & $2$ & $14$ \\ \hline
 $h^1(F(p))$ & $0$ & $0$ & $0$ & $3$ & $3$ & $0$ & $0$ \\ \hline
 $h^2(F(p))$ & $7$ & $4$ & $0$ & $0$ & $0$ & $0$ & $0$ \\ \hline
 $h^3(F(p))$ & $0$ & $0$ & $0$ & $0$ & $0$ & $0$ & $0$ \\ \hline
 \end{tabular}
 \caption{Cohomology table for $F\in \calr(-1,4,4)$ with spectrum $\{-1, -1, -1, -1\}$ as in the example. }
 \label{tab: cohomology 141}
\end{table}

\end{example}

\begin{example}[$c_3 = 4$, spectrum $\{-2, -1, -1, 0\}$]
Consider $C$ a disjoint union of a plane cubic and a twisted cubic, thus a curve of degree $6$ and genus $0$. Then, a general global section of $\omega_C(1)$ vanishes along $3$ points on each irreducible component. Thus $C$ corresponds to $F(2)$ such that $(c_1(F) , c_2(F), c_3(F)) = (-1, 4, 4)$ and $h^1(F(-1)) = h^1(\sI_C) = 1$, since $C$ has $2$ smooth connected components. Note also that $h^0(F(1)) = 0$ in this case. With the code below, we show that $\Ext^2(F, F)= 0 $ generically, and these sheaves belong to another generically smooth irreducible component of $\calr(-1, 4, 4)$.

\begin{lstlisting}[language=Macaulay2]
load "refshsetup.m2"
C1 = minors(2,matrix{{x_0,x_1,x_2},{x_1,x_2,x_3}});
C2 = ideal(x_0, random(3,R));
C = intersect(C1,C2); degree C, genus C
F = Serre(C,3); chern F
Ext^2(F,F) -- unobstructed
\end{lstlisting}

The cohomology table for these sheaves is as follows: 

\begin{table}[H]
 \centering
 \def\arraystretch{1.5}
 \begin{tabular}{|c|c|c|c|c|c|c|c|c|} \hline
 &$-3$ &$-2$ &$-1$ & $0$ & $1$ & 2 & $3$ \\ \hline \hline
 $h^0(F(p))$ & $0$ & $0$ & $0$ & $0$ & $0$ & $4$ & $14$ \\ \hline
 $h^1(F(p))$ & $0$ & $0$ & $1$ & $3$ & $3$ & $2$ & $0$ \\ \hline
 $h^2(F(p))$ & $7$ & $4$ & $1$ & $0$ & $0$ & $0$ & $0$ \\ \hline
 $h^3(F(p))$ & $0$ & $0$ & $0$ & $0$ & $0$ & $0$ & $0$ \\ \hline
 \end{tabular}
 \caption{Cohomology table for $F\in \calr(-1,4,4)$ with spectrum $\{-2, -1, -1, 0\}$ as in the example. }
 \label{tab: cohomology 142}
\end{table}

\end{example}

\begin{example}[$c_3 = 6$, spectrum $\{-2, -1, -1, -1\}$]
Let $Y$ be a disjoint union of $3$ lines, and let $C$ be the residual curve to the intersection of two general elements of $H^0(\sI_Y(3))$. Then $C$ is a smooth curve of degree $6$ and genus $1$, due to \ref{P1}, and 
\[
h^1(\sI_C(l)) = h^1(\sI_Y(2-l)) =\begin{cases}
2, & l=1, 2 \\ 0, & l\neq 1, 2 
\end{cases}.
\]
Then, by \ref{P3}, 
\[
h^0(\omega_C(1)) = h^1(\OO_C(-1)) = h^0(\sI_Y(3)) - h^0(\sI_X(3)) = 6, 
\]
where $X$ is the complete intersection linking $C$ and $Y$. Since $C$ is smooth, it corresponds to a sheaf $F(2)$ such that $(c_1(F) , c_2(F), c_3(F)) = (-1, 4, 6)$ and $h^0(F(1)) = h^1(F(-1)) = 0$. Hence, $F$ has spectrum $\{-2, -1, -1, -1\}$. These sheaves belong to a generically smooth irreducible component of $\calr(-1, 4, 6)$, as one can check with the code below.

\begin{lstlisting}[language=Macaulay2]
load "refshsetup.m2";
Y = dsLns 3; 
C = quotient(rand(Y, 3, 3), Y); degree C, genus C
smooth C -- check that C is smooth
F = Serre(C, 3); chern F
Ext^2(F, F) -- unobstructed
\end{lstlisting}

Also, the cohomology table is the following.

\begin{table}[H]
 \centering
 \def\arraystretch{1.5}
 \begin{tabular}{|c|c|c|c|c|c|c|c|c|} \hline
 &$-3$ &$-2$ &$-1$ & $0$ & $1$ & 2 & $3$ \\ \hline \hline
 $h^0(F(p))$ & $0$ & $0$ & $0$ & $0$ & $0$ & $3$ & $15$ \\ \hline
 $h^1(F(p))$ & $0$ & $0$ & $0$ & $2$ & $2$ & $0$ & $0$ \\ \hline
 $h^2(F(p))$ & $7$ & $5$ & $1$ & $0$ & $0$ & $0$ & $0$ \\ \hline
 $h^3(F(p))$ & $0$ & $0$ & $0$ & $0$ & $0$ & $0$ & $0$ \\ \hline
 \end{tabular}
 \caption{Cohomology table for $F\in \calr(-1,4,6)$ with spectrum $\{-2, -1, -1, -1\}$ as in the example. }
 \label{tab: cohomology 161}
\end{table}
    
\end{example}

\begin{example}[$c_3 =6$, spectrum $\{-2, -2, -1, 0\}$] Let $C$ be the disjoint union of two plane cubics, thus a curve of degree $6$ and genus $1$. Then a general global section of $\omega_C(1)$ has isolated zeros and $C$ corresponds to a reflexive sheaf $F(2)$ with Chern classes $(c_1(F) , c_2(F), c_3(F)) = (-1, 4, 6)$. In this case, $h^1(F(-1)) = h^1(\sI_C) = 1$. Also note that $h^0(F(1)) = 1$, but a section of $F(1)$ defines a curve of degree $4$ and genus $-2$ that would be harder to describe. With the code below, we check that these sheaves are generically unobstructed in $\calr(-1, 4, 6)$

\begin{lstlisting}[language=Macaulay2]
load "refshsetup.m2";
C = intersect(ideal(x_0, random(3, R)), ideal(x_1, random(3, R)));
degree C, genus C
F = Serre(C, 3); chern F
Ext^2(F, F) -- unobstructed
\end{lstlisting}

The cohomology table for these examples is as follows: 

\begin{table}[H]
 \centering
 \def\arraystretch{1.5}
 \begin{tabular}{|c|c|c|c|c|c|c|c|c|} \hline
 &$-3$ &$-2$ &$-1$ & $0$ & $1$ & 2 & $ 3$ & $4$\\ \hline \hline
 $h^0(F(p))$ & $0$ & $0$ & $0$ & $0$ & $1$ & $5$ & $16$ & $36$ \\ \hline
 $h^1(F(p))$ & $0$ & $0$ & $1$ & $2$ & $3$ & $2$ & $1$ & $0$\\ \hline
 $h^2(F(p))$ & $7$ & $5$ & $2$ & $0$ & $0$ & $0$ & $0$ & $0$ \\ \hline
 $h^3(F(p))$ & $0$ & $0$ & $0$ & $0$ & $0$ & $0$ & $0$ & $0$\\ \hline
 \end{tabular}
 \caption{Cohomology table for $F\in \calr(-1,4,6)$ with spectrum $\{-2, -2, -1, 0\}$ as in the example.}
 \label{tab: cohomology 162}
\end{table}

\end{example}

\section{Stable reflexive sheaves with \texorpdfstring{$c_1 = -1$}{c1 = -1} and \texorpdfstring{$c_3=8$}{c3=8}}\label{sec: odd-8}

To study $\calr(-1, 4, 8)$, we first observe that $h^0(F(1))\le1$, due to Lemma \ref{lem: h0F1>0}, and that $h^1(F(-1)) \leq 1$, due to the possible spectra in Table \ref{tab: spec c1=-1}. We then define the subsets
\[
\calr(-1, 4, 8)_{l, m} \coloneqq \left\{ F \in\calr(-1, 4, 8) \mid h^0(F(1))=l, \, h^1(F(-1)) = m \right\}, 
\]
where $l,m\in\{0,1\}$. Note that for $m=0$ the spectrum is $\{-2,-2,-1,-1\}$ and for $m=1$ it is $\{-3,-2,-1,0\}$.

\begin{theorem}\label{thm: -148}
The moduli scheme $(\calr(-1, 4, 8))_{\rm red}$ is the closure of $\calr(-1,4,8)_{0,0}$, hence, it is reduced and irreducible of expected dimension $27$. Moreover, 
\begin{itemize}
 \item $\calr(-1,4,8)_{0,0}$ is smooth, $\extd^2(F,F) = 0$ for every $F \in \calr(-1,4,8)_{0,0}$ ;
 \item $\calr(-1,4,8)_{0,1} = \emptyset$;
 \item $\calr(-1, 4, 8)_{1, 1}$ is irreducible of dimension $26$, and $\extd^2(F,F) = 1$ for any $F\in \calr(-1, 4, 8)_{1, 1}$;
 \item $\calr(-1, 4, 8)_{1, 0} = \mathcal{G}_1 \cup \mathcal{G}_2$, both components have dimension $25$; and also $\extd^2(F,F) = 0$ for every $F\in \mathcal{G}_2$, or a generic $F\in \mathcal{G}_1$.
\end{itemize}
\end{theorem}

\begin{table}[H]
 \centering
 \def\arraystretch{1.5}
 \begin{tabular}{|c|c|c|c|c|c|c|c|c|} \hline
 &$-3$ &$-2$ &$-1$ & $0$ & $1$ & 2 & $3$ & $4$ \\ \hline \hline
 $h^0(F(p))$& $0$ & $0$ & $0$ & $0$ & $l$ & $lk +m+4$ & $16$ &$37$\\ \hline
 $h^1(F(p))$ & $0$ & $0$ & $m$ & $m+1$ & $l+1$ & $lk +m$ & $m$& $0$\\ \hline
 $h^2(F(p))$ & $9$ & $6$ & $m+2$ & $m$ & $0$ & $0$ & $0$ & $0$\\ \hline
 $h^3(F(p))$& $0$ & $0$ & $0$ & $0$ & $0$ & $0$ & $0$& $0$ \\ \hline
 \end{tabular}
 \caption{Cohomology table for $F\in \calr(-1,4,8)$, where $l,m,k \leq 1$ }
 \label{tab: cohomology 181}
\end{table}

\subsection{Description of \texorpdfstring{$\calr(-1, 4, 8)_{0, 0}$}{R(-1, 4, 8) 0, 0}} We start by considering the case $h^0(F(1)) = 0$ and $h^1(F(-1)) = 0$. In particular, the spectrum of $F$ is $\{ -2, -2, -1, -1\}$. Then $F(2)$ corresponds to a curve $C$ of degree $6$ and genus $2$ such that: 
\begin{align*}
 h^0(\sI_C(2)) & = h^1(\sI_C) = h^2(\sI_C(1)) = 0,\\
 h^1(\sI_C(1)) & = h^1(\sI_C(2)) = 1, \text{ and}\\
 h^0(\sI_C(3)) & = h^0(F(2))-1 = 3 + h^1(F(2))\geq 3.
\end{align*}
Then, we can prove the following lemma.

\begin{lemma}\label{lem: curve-1800}
For a curve $C$ as above, we have a resolution of the form
\begin{equation}\label{eq: res-1800}
 0 \longrightarrow \op3(-6) \longrightarrow \begin{matrix} \op3(-5)^{\oplus 3} \\ \oplus \\ \op3(-4)^{c+1} \end{matrix} \longrightarrow \begin{matrix} \op3(-3)^{\oplus 3} \\ \oplus \\ \op3(-4)^{c+1}\end{matrix} \longrightarrow \sI_C \longrightarrow 0
\end{equation}
where $c \in \{0,1\}$ and the first map on the left is given, up to a choice of coordinates, by $\begin{bmatrix} x_0,x_1, x_2, x_3^2, 0^c\end{bmatrix}^T$. Moreover, 
\begin{enumerate}
 \item a general curve with $c=0$ is smooth and irreducible; and
 \item if $c=1$ then $C$ is the limit of curves with resolution \eqref{eq: res-1800} with $c=0$.
\end{enumerate}
\end{lemma}

\begin{proof}
Applying Lemma \ref{lem: curv6cub} we have that either 
\begin{enumerate}
 \item $C$ is directly linked to a curve of degree $3$ and genus $-1$;
 \item $C$ has a planar subcurve with an extremal residual curve of degree $2$ or $3$.
\end{enumerate}
In the first case, we note that every curve $Y$ of degree $3$ and genus $-1$ is extremal, cf. \cite[Theorem 4.1]{MDP2}. The Rao module of $Y$ is then $M_Y \cong R/(x_0, x_1, x_2, x_3^2)$, for this we only need to consider $Y$ a disjoint union of a conic $V(x_2, x_3^2 -q(x_0, x_1))$ and a line $V(x_0, x_1)$. Then the Rao module of $C$ is $M_C = \Ext^4_R(M_Y, R(-6)) = M_Y(-1)$ which has the Koszul resolution
\[
0 \to R(-6) \to 
 \begin{matrix}R(-4) \\ \oplus \\ R(-5)^{\oplus 3} \end{matrix}
\to 
\begin{matrix}R(-3)^{\oplus 3}\\ \oplus \\R(-4)^{\oplus 3} \end{matrix}
\to 
\begin{matrix}R(-2)^{\oplus 3}\\ \oplus \\R(-3)\end{matrix}
\to R(-1) \to M_C \to 0
\]
From \cite[II \S5]{MDP} we can show that $\sI_C$ has a resolution of the form 
\[
0 \longrightarrow \op3(-6) \overset{\mbox{\normalsize B}}{\longrightarrow}\begin{matrix} \op3(-5)^{\oplus 3} \\ \oplus \\ \op3(-4) \\ \oplus \\ L \end{matrix} \overset{\mbox{\normalsize A}}{\longrightarrow} \begin{matrix} \op3(-3)^{\oplus 3} \\ \oplus \\ \op3(-4) \\ \oplus \\ L \end{matrix} \longrightarrow \sI_C \longrightarrow 0
\]
where $L$ is {a direct sum of line bundles}, and the map from $\op3(-6)$ to $L$ is zero. Indeed, this computation relies on knowing the postulation character $\gamma_C$, which can be computed from $Y$ and the property \ref{P3}; it is $\gamma_C = \{-1, -1, -1, 2, 2, -1\}$. On the other hand, imposing minimality and removing trivial factors, we conclude that $L = \op3(-4)^{\oplus c}$, for $c\geq 0$. Choosing coordinates, we may assume that $B = \begin{bmatrix} x_0 & x_1 & x_2 & x_3^2& 0\, ^c \end{bmatrix}^T$. Since $AB = 0$, the rows of $A$ are syzygies of $B$, hence we may write
\[
A = \begin{bmatrix}
 A_{11}^{3\times 3} & A_{12}^{3\times c+1} \\
 A_{21}^{c+1\times 3} & 0 ^{c +1 \times c+1}
\end{bmatrix},
\]
where the formats of the blocks are indicated in superscript. Moreover, the degrees of the entries are as follows: $\deg A_{11} = 2$, and $\deg A_{12} =\deg A_{21} = 1$. We have that $A_{21}$ has generic rank $\leq 2$ since $A_{21} \begin{bmatrix}x_0& x_1&x_2\end{bmatrix}^T = 0$. Note that $A$ must have generic rank $3+c$. On the other hand, at a general point $p$ such that $\det A_{11} (p) \neq 0$, we have $\rk A(p) = 3 + \rk( A_{21}(p)A_{11}(p)^{-1}A_{12} (p)) \leq 5$, hence $c\leq 2$. We observe further that $A$ drops rank along the variety defined by the maximal minors of $A_{12}$. This eliminates the case $c=2$, because $A$ would have rank $\leq 4$ along the cubic surface $V(\det A_{12})$. Moreover, for $c=1$, the curve $C$ does not link directly to a curve $Y$ as above. Indeed, the minors of $A_{12}$ define a (possibly degenerate) twisted cubic, the residual curve in $C$ is a plane cubic, and the linear form defining this plane divides every element in $H^0(\sI_C(3))$. 

In the second case, we have that $C$ must fit in the following exact sequence:
\[
0\longrightarrow \sI_Y(-1) \longrightarrow \sI_C \longrightarrow \sI_{Z/H}(-d) \longrightarrow 0,
\]
with $d \in \{3,4\}$, and $Y$ an extremal curve of degree $6-d$. From the resolutions of $\sI_Y$ and $\sI_{Z/H}$ one shows that $\sI_C$ has a resolution of the form \eqref{eq: res-1800} with $c=1$. Since $h^2(\sI_C(1)) = 0$ and $h^3(\sI_Y) = h^3(\op3) = 0$, we have that 
\[
h^2(\OO_H(1-d)) = h^2(\sI_{Z/H}(1-d)) = 0.
\]
Therefore, $d < 4$, and the only possibility is $d = 3$. {It follows that $p_a(Y) = 0$ and $h^0(\OO_Z) =1$, cf. Lemma \ref{lem: curv6cub}. }

To conclude, we observe that if $Y$ is a disjoint union of a conic and a line, then $\sI_Y(3)$ is globally generated. By \cite[Théorème 5.1]{MDP2}, a general curve linked to $Y$ is smooth; it is also irreducible since $h^1(\sI_C) = 0$. Furthermore, the curves of Lemma \ref{lem: curve-1800} for which $c=1$ are flat limits with constant Rao module of those satisfying $c=0$, cf. \cite[Theorem 4.1]{Kleppe-curves}. 
\end{proof}

A general curve $C$ from Lemma \ref{lem: curve-1800} is linked to a curve $Y$ of degree $3$ and genus $-1$. We have that $\hilb^{3,-1}(\p3)$ is irreducible of dimension $12$. Then, we have that the curves from Lemma \ref{lem: curve-1800} form an irreducible family $\mathcal{G}$ of dimension
\[
\dim \mathcal{G} = 12 + 2h^0(\sI_Y(3)) - 2h^0(\sI_C(3)) = 12 + 18 -6 = 24.
\]
Since a general curve is smooth and irreducible, $h^0(\omega_C(1)) = 7$, by Riemann-Roch, and there exists $\xi \in H^0(\omega_C(1))$ with isolated zeros. Thus $\calr(-1,4,8)_{0,0}$ is irreducible of dimension 
\[
\calr(-1,4,8)_{0,0} = 24 + 7 - h^0(F(2)) = 27.
\]
Moreover, it also follows from Lemma \ref{lem: curve-1800} that $F\in \calr(-1,4,8)_{0,0}$ is $3$-regular with $h^0(F) = h^1(F(-1)) =0$. Then $\extd^2(F,F) = 0$, by Lemma \ref{lem: ext2=0}, hence $\calr(-1,4,8)_{0,0}$ is smooth of expected dimension; its closure is an irreducible component of $\calr(-1,4,8)$.

\subsection{Description of \texorpdfstring{$\calr(-1, 4, 8)_{0, 1}$}{R(-1, 4, 8) 0, 1}} 
Now consider the case $h^0(F(1)) = 0$ and $h^1(F(-1)) = 1$, hence, $F$ has spectrum $\{ -3, -2, -1, 0\}$. Then $F(2)$ must correspond to a curve $C$ of degree $6$ and genus $2$ such that
\[
h^0(\sI_C(2)) = 0, \, h^1(\sI_C(1)) = 2, \, \text{and } h^0(\sI_C(3)) = h^0(F(2))-1 \geq 3.
\]
By Lemma \ref{lem: curv6cub}, in view of our description of $\calr(-1, 4, 8)_{0, 0}$, we see that $\sI_C$ fits in an exact sequence
\[
0\longrightarrow \sI_Y(-1) \longrightarrow \sI_C \longrightarrow \sI_{Z/H}(-4) \longrightarrow 0
\]
with $\deg(Y) = 2$, $p_a(Y) = -1$, and $h^0(\OO_Z)= 1$. Dualizing and twisting the sequence above, we have {the exact sequence}
\[
0 \longrightarrow \OO_D(2) \longrightarrow \omega_C(1) \overset{\pi}{\longrightarrow} \OO_Y \longrightarrow \OO_Z \longrightarrow 0
\]
It follows that any global section of $\omega_C(1)$ vanishes along a line $L\subset Y$. If $Y = L \cup L'$, with $Z\in L$, then the image of $\pi$ is $\OO_{L'}\oplus \OO_L(-Z)$. If $Y$ is a double line, then $\OO_Y \to \OO_Z$ factors through $\OO_L \to \OO_Z$. Hence, the image of $\pi$ is an extension of $\OO_L(-Z)$ by $\OO_L$. Therefore, any sheaf $F$ as above is torsion-free but not reflexive and 
\[
\calr(-1, 4, 8)_{0, 1} = \emptyset.
\]

\subsection{Description of \texorpdfstring{$\calr(-1, 4, 8)_{1, m}$}{R(-1, 4, 8) 1, m}} Now we turn our attention to the cases where $h^0(F(1)) = 1$ and $h^1(F(-1)) = m \leq 1$. Then $F(1)$ corresponds to a curve of degree $4$ and genus $-1$ such that 
\[
h^1(\sI_C(-1)) = m, \,\, h^1(\sI_C) = m+1, \,\, h^1(\sI_C(1)) = 2, \,\, \text{and} \,\, h^1(\sI_C(2)) = h^0(\sI_C(2)).
\]
Curves of degree $4$ and genus $-1$ were classified in \cite[Proposition 6.1]{NS-deg4}. We have recalled their classification in \S\ref{sec: 048-1m} but state it again for convenience. The corresponding Hilbert Scheme has $3$ irreducible components whose general members are extremal curves ($h^1(\sI_C(2)) = 2$), disjoint union of two conics ($h^1(\sI_C(2))= 1$), and disjoint union of twisted cubic and a line ($h^1(\sI_C(2)) = 0$). In particular, $m=1$ if and only if $C$ is extremal.

If $C$ is extremal, then $C$ is a curve in a double plane given by {the exact sequence}
\[
0 \longrightarrow \sI_L(-1) \longrightarrow \sI_C \longrightarrow \sI_{Z/H}(-3) \longrightarrow 0
\]
where $L$ is a line and $h^0(\OO_Z) = 2$. Dualizing and twisting this sequence, we obtain {the exact sequence}
\[
0\longrightarrow \OO_D(3) \longrightarrow \omega_C(3) \longrightarrow \OO_L \longrightarrow 0,
\]
where $D\subset C\cap H$ is a plane cubic. Thus, $h^0(\omega_C(1) = 10$ and global sections with isolated zeros. Therefore, $\calr(-1, 4, 8)_{1, 1}$ is irreducible of dimension
\[
\dim \calr(-1, 4, 8)_{1, 1} = 17 + 10 - 1 = 26.
\]
In particular, it cannot be an irreducible component of $\calr(-1, 4, 8)$. Due to Proposition \ref{prop: extremal-ext}, $\extd^2(F,F) = 1$ for any $F\in \calr(-1, 4, 8)_{1, 1}$.

If $C$ is subextremal, then $\sI_C$ fits in an exact sequence (cf. \cite[Proposition 9.9]{HS})
\[
0 \longrightarrow \sI_Y(-1) \longrightarrow \sI_C \longrightarrow \sI_{Z/H}(-2) \longrightarrow 0,
\]
where $Y$ is a plane conic and $h^0(\OO_Z) = 2$. Dualizing and twisting this sequence, we have:
\[
0 \longrightarrow \OO_{Y'}(2) \longrightarrow \omega_C(3) \longrightarrow \OO_Y(2) \longrightarrow 0,
\]
where $Y' \subset C\cup H$ is a conic. Thus, $h^0(\omega_C(3)) = 10$ and a general global section has isolated zeros. Then the sheaves corresponding to these curves form an irreducible variety $\mathcal{G}_1 \subset \calr(-1, 4, 8)_{1, 0}$ of dimension
\[
\dim \mathcal{G}_1 = 16 + 10 - 1 = 25.
\]

Reasoning as we did for extremal curves in Section \ref{sec: extremal}, we find
\[
\Ext^2(F,F) \cong Ext^2(F,\sI_{Z/H}(-2)).
\]
Moreover, $\Ext^2(F, \OO_H(-2)) = 0$. Then $\extd^2(F, F) \leq \extd^2(F, \OO_Z)$, but $\extd^2(F, \OO_Z)$ depends on how $Z$ intersect the points where $F$ is not free, i.e., the vanishing locus of the extension class $\xi \in H^0(\omega_C(3))$ defining $F$. When $C$ is a disjoint union of two conics, a generic choice of $\xi$ does not vanish on $Z$. The general case follows from a direct but lengthy computation similar to the proof of Lemma \ref{lem: vanish-extremal}. With the code below, one can produce unobstructed examples.

\begin{lstlisting}[language=Macaulay2]
load "refshsetup.m2";
C =intersect(ideal(x_0, random(2,R)), ideal(x_1, random(2,R)));
F = Serre(C,1); chern F
Ext^2(F,F) --unobstructed
\end{lstlisting}

\begin{remark}
 Some examples with $\extd^1(F,\OO_Z)>0$ turned out to be unobstructed. However, we do not know whether this is true for any sheaf in $\calr(-1, 4, 8)_{1, 0}$ corresponding to a subextremal curve.
\end{remark}

If $C$ is neither extremal nor subextremal, then the Rao module of $C$ is isomorphic to 
\[
\frac{\kappa[x_0,x_1,x_2,x_3]}{(x_0,x_1,x_2^2,x_2x_3,x_3^2)},
\]
cf. \cite[Proposition 4.2]{NS-deg4}. In particular, we have that $h^1(\sI_C) = 1$, $h^1(\sI_C(1)) = 2$, and $h^1(\sI_C(t)) = 0$, for $t\not\in \{0,1\}$. A reflexive sheaf $F(1)$ corresponding to $C$ must be $2$-regular, i.e., $F$ is $3$-regular. Also, we have that $h^1(F(-1)) = h^1(\sI_C(-1)) = 0$ and $h^0(F) = 0$. Then $\Ext^2(F,F) = 0$ by Lemma \ref{lem: ext2=0}.

These curves form an irreducible variety of dimension $16$. Generically, $C$ is a disjoint union of a twisted cubic $Y$ and a line $L$. Hence, $\omega_C(3) = \omega_Y(3)\oplus \omega_L(3)$ has global sections vanishing only in dimension zero. Then the sheaves corresponding to these curves form an irreducible variety $\mathcal{G}_2 \subset \calr(-1, 4, 8)_{1, 0}$ of dimension
\[
\dim \mathcal{G}_2 = 16 + h^0(\omega_C(3)) -h^0(F(1)) = 16 + 10 -1 = 25.
\]

We conclude that 
\[
\calr(-1, 4, 8)_{1, 0} = \mathcal{G}_1 \cup \mathcal{G}_2
\]
has dimension $25$. 

\section{Stable reflexive sheaves with \texorpdfstring{$c_1 = -1$}{c1 = -1} and \texorpdfstring{$c_3=10$}{c3=10}}\label{sec: odd-10}

Now we turn to reflexive sheaves with Chern classes $(c_1, c_2, c_3) = (-1, 4, 10)$. Due to Lemma \ref{lem: h0F1>0}, such a sheaf $F$ must satisfy $h^0(F(1))\le 1$, while the possible spectra imply that $h^1(F)\le1$. Therefore, we introduce the following sets
\[
\calr(-1, 4, 10)_{l, m} \coloneqq \{\, F \in\calr(-1, 4, 10) \mid h^0(F(1))=l , ~h^1(F)=m \,\}, 
\]
so that 
\[
\calr(-1, 4, 10) ~ = \bigsqcup_{l, m=0, 1} \calr(-1, 4, 10)_{l, m}; 
\]
note that the spectrum only depends on $m$, it is $\{-2, -2, -2, -1\}$ for $m=0$, and $\{-3, -2, -1, -1\}$ for $m=1$. We then prove: 

\begin{theorem}\label{thm: -1,10}
The moduli scheme $\calr(-1, 4, 10)$ is an integral, unirational quasi-projective variety of dimension 27, with a general element lying in $\calr(-1, 4, 10)_{0, 0}$; it contains $\calr(-1, 4, 10)_{1, 0}$, $\calr(-1, 4, 10)_{0, 1}$, and $\calr(-1, 4, 10)_{1, 1}$ as divisors. Moreover, $\calr(-1, 4, 10)$ is smooth away from a set of codimension at least two.
\end{theorem}

Moreover, the cohomology table for these sheaves is as follows.
\begin{table}[H]
\centering
\def\arraystretch{1.5}
\begin{tabular}{|c|c|c|c|c|c|c|c|c|} \hline
 &$-3$ &$-2$ &$-1$ & $0$ & $1$ & 2 & $3$ \\ \hline \hline
 $h^0(F(p))$ & $0$ & $0$ & $0$ & $0$ & $l$ & $5+k$ & $17$ \\ \hline
 $h^1(F(p))$ & $0$ & $0$ & $0$ & $m$ & $l$ & $k$ & $0$ \\ \hline
 $h^2(F(p))$ & $10$ & $7$ & $3$ & $m$ & $0$ & $0$ & $0$ \\ \hline
 $h^3(F(p))$ & $0$ & $0$ & $0$ & $0$ & $0$ & $0$ & $0$ \\ \hline
\end{tabular}
\caption{Cohomology table for $F\in\calr(-1, 4, 10)_{lm}$; $k=l\cdot m$. }
\label{tab: cohomology 1101}
\end{table}


\subsection{Description of \texorpdfstring{$\calr(-1, 4, 10)_{0, m}$}{R(-1, 4, 8) 0, m}} 

Suppose that $h^0(F(1))=0$ and consider a curve $C$ given by the vanishing locus of a section of $F(2)$. Then {the sequence}
\[
0 \longrightarrow \op3 \longrightarrow F(2) \longrightarrow \sI_C(3) \longrightarrow 0
\]
{is exact} and $C$ is a curve of degree $6$ and genus $3$. We have that $h^0(\sI_C(2)) = h^0(F(1))=0$ and $h^0(\sI_C(3)) = h^0(F(2)) - 1 \geq 4$. Then Lemma \ref{lem: curv6cub} tells us that either $C$ is linked to a curve $Y$ of degree $3$ and genus $0$, i.e., a (possibly degenerate) twisted cubic, or there exists a hyperplane $H$ such that $C$ fits into {the exact sequence}
\begin{equation}\label{seq: ideal1.10}
 0 \longrightarrow \sI_Y (-1) \overset{h}{\longrightarrow} \sI_C \longrightarrow \OO_H(-4) \longrightarrow 0 
\end{equation}
where $Y$ is a curve of degree $2$ and genus $-1$. We have $h^1(\sI_C(3))= 0$ in both cases. Therefore, $h^1(F(2)) = 0$ and $F$ is $3$-regular. Noting that $h^0(F)=0$ (by stability) and $h^1(F(-1))=0$ (due to the spectra), we apply Lemma \ref{lem: ext2=0} to conclude that every $F\in\calr(-1, 4, 10)_{0m}$ satisfies $\extd^2(F, F)=0$. Moreover, $h^1(F) = 1$ if and only if it corresponds to a curve given by $\eqref{seq: ideal1.10}$.

For $F\in\calr(-1, 4, 10)_{00}$, the corresponding curve is ACM with resolution
\[
0\longrightarrow \op3(-4)^{\oplus 3} \longrightarrow \op3(-3)^{\oplus 4} \longrightarrow \sI_C \longrightarrow 0.
\]
Therefore, $F$ has a resolution of the form
\[
0\longrightarrow \op3(-3)^{\oplus 3} \longrightarrow \op3(-2)^{\oplus 5} \longrightarrow F \longrightarrow 0
\]
We then apply Lemma \ref{lem: lift} to conclude that $\calr(-1, 4, 12)_{00}$ is smooth and unirational of dimension $27$, and its closure is an irreducible component of $\calr(-1, 4, 10)$. 

For $F\in\calr(-1, 4, 10)_{01}$, we use the sequence \eqref{seq: ideal1.10} to conclude that $C$ lies on a $23$-dimensional family, a divisor on the Hilbert Scheme $\hilb^{6, 3}(\p3)$, and also that $h^0(\omega_C(1)) = 8$. Then
\[
\dim \calr(-1, 4, 10)_{01} = 23 + 8 - h^0(F(2)) = 26.
\]


\subsection{Description of \texorpdfstring{$\calr(-1, 4, 10)_{1, m}$}{R(-1, 4, 10) 1, m}} 
Now we suppose that $h^0(F(1))=1$ and let $C$ be a curve given by 
\begin{equation}\label{eq: r_1m}
0 \longrightarrow \op3(-1) \longrightarrow F \longrightarrow \sI_C \longrightarrow 0.
\end{equation}
Then $C$ has degree $4$ and genus $0$. The Hilbert Scheme $\hilb^{4, 0}(\p3)$ has two irreducible components, denote them by $\mathcal{L}$ and $\mathcal{E}$. The general members are, respectively: 
\begin{enumerate}
\item[(1)] $C$ is a rational quartic;
\item[(2)] $C$ is a disjoint union of a plane cubic and a line;
\end{enumerate}
cf. \cite[\S 4]{MDP2}. The distinction is that $\mathcal{E}$ is composed by extremal curves, i.e., $h^0(\sI(2)) \geq 2$, while for a general rational quartic, we have $h^0(\sI(2)) = 1$. However, these two components intersect, cf. \cite[Example 3.10]{N-deg3}. Note that $C\in \mathcal{E}$ if and only if $h^0(\sI_C(2)) \geq 2$, if and only if $h^1(\sI_C)=1$.

If $F\in\calr(-1, 4, 10)_{10}$, then $h^1(\sI_C)=0$, so $C\in\mathcal{L} \setminus \mathcal{E}$. Then, $h^0(\sI_C(2)) = 1$ which implies $h^1(F(2)) = 0$. Due to the spectrum we have $h^2(F(1)) = h^1(F(-1)) = 0$, and it follows that $F$ is $3$-regular. Then Lemma \ref{lem: ext2=0} implies that $\extd^2(F, F) =0$. 

For a general element $C$ in $\mathcal{L}$, one can check that $h^0(\omega_C(3))=11$; using the formula in display \eqref{eq: dimform-red}, we conclude that
\[ 
\dim\calr(-1, 4, 10)_{10} = 16+11-1 = 26; 
\]
therefore, $\calr(-1, 4, 10)_{10}$ is not an irreducible component of $\calr(-1, 4, 10)$.

Similarly, if $F\in\calr(-1, 4, 10)_{11}$, then $h^1(\sI_C)=1$, so $C\in\mathcal{E}$. We have that 
\begin{equation}\label{eq: extr1}
 0\longrightarrow \sI_L(-1) \longrightarrow \sI_C \longrightarrow \sI_{P/H}(-3) \longrightarrow 0
\end{equation}
{is an exact sequence,} where $P\in H$ is a point and $L$ is a line. Dualizing this sequence, we get
\[
0 \longrightarrow \OO_Y(3) \longrightarrow \omega_C(3) \longrightarrow \OO_L(1) \longrightarrow 0
\]
where $Y\subset C\cup H$ is a plane cubic; it follows that $h^0(\omega_C(3))=11$. Using the formula in display \eqref{eq: dimform-red}, we conclude that
\[ 
\dim\calr(-1, 4, 10)_{11} = 16+11-1 = 26; 
\]
therefore, $\calr(-1, 4, 10)_{11}$ is not an irreducible component of $\calr(-1, 4, 10)$ either.

We know from Proposition \ref{prop: extremal-ext} that $\extd^2(F, F)=0$, unless the extension class $\xi$ defining $F$ vanishes at $P$. In this case, $\extd^2(F, F) \leq 1$. Since $P$ imposes only one condition on the extension $\xi$, this potentially obstructed locus has codimension $2$ in $\calr(-1, 4, 10)$. We give below a script to produce one obstructed example.

\begin{lstlisting}[language=Macaulay2]
load "refshsetup.m2";
l1 = random(1,R);
l2 = random(1,R);
f = random(3,R);
A = map(R^{2: -4,2: -2,-1}, R^{3: -5,-3},
 matrix{{-x_2,0,x_0,0}, 
	 {-x_1,x_0,0,0},
	 {0,-x_3^3,0,x_1},
	 {-x_3^3, -x_1^2*x_2,x_1^3,x_0},
	 {0, f*x_2, -f*x_1 -x_2*x_3^3, x_2^2}})
F = sheaf coker A;
chern(F) -- check the Chern classes
codim ann sheafExt^1(F,O) >= 3 -- check it is reflexive
codim ann sheafExt^2(F,O) >= 4 
codim ann sheafExt^3(F,O) >= 4 
Ext^2(F,F)
\end{lstlisting}


\section{Stable reflexive sheaves with \texorpdfstring{$c_1 = -1$}{c1 = -1} and \texorpdfstring{$c_3=12$}{c3=12}}\label{sec: odd-12}

To describe $\calr(-1,4,12)$, we start by noting that, in this case, every sheaf fits into a single family.

\begin{proposition}\label{prop: res-odd-12}
Let $F\in \calr(-1, 4, 12)$ then $F$ has a resolution of the form
\[
0 \longrightarrow \begin{matrix}
 \op3(-4) \\ \oplus \\ \op3(-3)
\end{matrix} \longrightarrow \begin{matrix}
 \op3(-1) \\ \oplus \\ \op3(-2)^{\oplus 2} \\ \oplus \\ \op3(-3)
\end{matrix} \longrightarrow F \longrightarrow 0,
\]
which is not necessarily minimal. 
\end{proposition}

\begin{proof}
We have that $h^0(F(1))\geq 1$. Consider $C$ a curve given by {an exact sequence}
\begin{equation}\label{seq: ideal1.12}
 0 \longrightarrow \op3 \longrightarrow F(1) \longrightarrow \sI_C(1) \longrightarrow 0 .
\end{equation}
Then $C$ is a curve of degree $4$ and genus $1$. Such curves are ACM due to \cite[Proposition 3.5]{Hart-spacecurves}, hence $H^1_*(F) = 0$. It follows that $h^0(\sI_C(2)) = 2$. Let $Q_1, Q_2 \in H^0(\sI_C(2))$ linearly independent. If $\gcd(q_1, q_2) = 1$ then $C$ is a complete intersection. If $q_1$ and $q_2$ share a common factor, then, up to a linear change of coordinates, we may assume that $q_2 = x_0x_1$ and either $q_1 = x_0x_2$ or $q_1 = x_0^2$. If $q_1 = x_0x_2$ then $C$ is the union of a cubic $Y$ in the plane $H = V(x_0)$ and the line $L= V(x_1,x_2)$ meeting at one point. Thus, 
\[
I_C = (x_1, x_2) \cap (x_0, x_1p_1+ x_2p_2) = (x_0x_1, x_0x_2, x_1p_1+ x_2p_2),
\]
where $\deg p_1 = \deg p_2 = 2$. If $q_1 = x_0^2$, $C$ is a curve in the double plane, then {we have an exact sequence}
\[
0 \longrightarrow \sI_Y(-1) \overset{\cdot x_0}{\longrightarrow} \sI_C \longrightarrow \OO_H(-3) \longrightarrow 0 
\]
where {$Y = V(x_0, x_1)$}. In any case, we get a resolution
\[
0 \longrightarrow \begin{matrix}
 \op3(-4) \\ \oplus \\ \op3(-3)
\end{matrix} \longrightarrow \begin{matrix}
 \op3(-2)^{\oplus 2} \\ \oplus \\ \op3(-3)
\end{matrix} \longrightarrow \sI_C \longrightarrow 0,
\]
which is minimal unless $C$ is a complete intersection. Combining this resolution with \eqref{seq: ideal1.12}, we conclude the proof.
\end{proof}

We then obtain the following cohomology table for $F\in \calr(-1, 4, 12)$.

\begin{table}[H]
 \centering
 \def\arraystretch{1.5}
 \begin{tabular}{|c|c|c|c|c|c|c|c|c|} \hline
 &$-3$ &$-2$ &$-1$ & $0$ & $1$ & 2 & $ 3$ \\ \hline \hline
 $h^0(F(p))$ & $0$ & $0$ & $0$ & $0$ & $1$ & $6$ & $18$ \\ \hline
 $h^1(F(p))$ & $0$ & $0$ & $0$ & $0$ & $0$ & $0$ & $0$ \\ \hline
 $h^2(F(p))$ & $11$ & $8$ & $4$ & $1$ & $0$ & $0$ & $0$ \\ \hline
 $h^3(F(p))$& $0$ & $0$ & $0$ & $0$ & $0$ & $0$ & $0$ \\ \hline
 \end{tabular}
 \caption{Cohomology table for $F\in \calr(-1,4,12)$ with spectrum $\{-3, -2, -2, -1\}$. }
 \label{tab: cohomology 1121}
\end{table}

Note that every $F$ in $\calr(-1, 4, 12)$ is $3$-regular and satisfies $h^0(F)=h^1(F(-1))=0$, therefore $\extd^2(F, F)=0$ by Lemma \ref{lem: ext2=0}, hence $\extd^1(F, F)=27$ and $\calr(-1, 4, 12)$ is smooth. On the other hand, Proposition \ref{prop: res-odd-12} implies that we have a surjective rational map $\p{}(\Hom(E, L)) \dashrightarrow \calr(-1, 4, 12)$, where 
\[ 
E\coloneqq \op3(-4)\oplus\op3(-3) ~~{\rm and} ~~ L\coloneqq \op3(-3)\oplus\op3(-2)^{\oplus 2}\oplus\op3(-1). 
\]
Hence, $\calr(-1, 4, 12)$ must be irreducible and unirational. We have therefore proved:

\begin{theorem}\label{thm: -1412}
$\calr(-1, 4, 12)$ is a smooth, irreducible, unirational quasi-projective variety of dimension 27. 
\end{theorem}


\end{document}